\documentclass[11pt,a4paper,english,reqno]{amsart}
\usepackage{amsmath,amssymb,amsfonts,epsfig,mathrsfs}
\usepackage[T1]{fontenc}

\usepackage{color}
\usepackage{array}
\usepackage{amsthm}
\usepackage{amstext}
\usepackage{graphicx}
\usepackage{setspace}
\usepackage[margin=2.5cm]{geometry}
\usepackage{bbm}
\usepackage{color}
\usepackage{enumitem}
\usepackage{undertilde}
\allowdisplaybreaks[4]

\usepackage{amscd,psfrag}
\usepackage{yhmath}
\usepackage[mathscr]{eucal}

\usepackage{slashed}

\makeatletter
\pdfpageheight\paperheight
\pdfpagewidth\paperwidth

\usepackage{epstopdf}
\usepackage{chngcntr}
\counterwithin{figure}{section}
\usepackage{mathrsfs}

\usepackage{indentfirst}	

\usepackage[normalem]{ulem}
\theoremstyle{plain}

\newtheorem{definition}{Definition}[section]
\newtheorem{theorem}[definition]{Theorem}
\newtheorem*{theorem*}{Theorem}

\newtheorem*{remark*}{Remark}
\newtheorem*{sideremark*}{Side Remark}\newtheorem*{mt*}{Main Theorem}

\newtheorem*{claim*}{Claim}
\newtheorem*{q*}{Question}
\newtheorem{lemma}[definition]{Lemma}
\newtheorem{corollary}[definition]{Corollary}
\newtheorem*{corollary*}{Corollary}
\newtheorem*{proposition*}{Proposition}

\newtheorem{proposition}[definition]{Proposition}

\newcommand{\R}{\mathbb{R}}

\newcommand{\na}{\nabla}

\newcommand{\dd}{{\rm d}}
\newcommand{\p}{\partial}
\newcommand{\e}{\epsilon}
\newcommand{\emb}{\hookrightarrow}

\newcommand{\map}{\rightarrow}
\newcommand{\G}{\Gamma}
\newcommand{\M}{\mathcal{M}}

\newcommand{\F}{\mathcal{F}}
\newcommand{\sph}{\mathbb{S}}
\newcommand{\dv}{{\rm div}}
\newcommand{\gl}{\mathfrak{gl}}
\newcommand{\defo}{\mathbb{D}(\na u)}
\newcommand{\id}{{\rm Id}}
\newcommand{\OO}{\mathcal{O}}
\newcommand{\fs}{{\mathfrak{S}}}
\newcommand{\hr}{{\widehat{\rho}}}
\newcommand{\hu}{{\widehat{u}}}

\newcommand{\hD}{{\widehat{\Delta}}}

\newcommand{\hg}{\widehat{\gamma}}

\newcommand{\lift}{\mathfrak{L}}
\newcommand{\geu}{g_{\rm Eucl}}

\newcommand{\haus}{\mathcal{H}}
\newcommand{\hdv}{{\widehat{\rm div}}}
\newcommand{\hna}{\widehat{\na}}
\newcommand{\TT}{\mathfrak{T}}
\newcommand{\hx}{\widehat{X}}
\newcommand{\nn}{{\mathbf{n}}}
\newcommand{\clx}{\widetilde{X}}
\newcommand{\proj}{{\rm pr}}
\newcommand{\E}{\mathcal{E}}
\newcommand{\RR}{\mathcal{R}}
\newcommand{\urho}{{\overline{\rho}}}
\newcommand{\lrho}{{\underline{\rho}}}
\newcommand{\one}{{\rm I}}
\newcommand{\two}{{\rm II}}
\newcommand{\three}{{\rm III}}
\newcommand{\four}{{\rm IV}}
\newcommand{\five}{{\rm V}}
\newcommand{\vis}{{\Big(\eta+\frac{N-2}{N}\mu\Big)}}
\newcommand{\huu}{{(\id-\proj)(U_\e-u)}}

\allowdisplaybreaks[4]

\def\Xint#1{\mathchoice
{\XXint\displaystyle\textstyle{#1}}%
{\XXint\textstyle\scriptstyle{#1}}%
{\XXint\scriptstyle\scriptscriptstyle{#1}}%
{\XXint\scriptscriptstyle\scriptscriptstyle{#1}}%
\!\int}
\def\XXint#1#2#3{{\setbox0=\hbox{$#1{#2#3}{\int}$ }
\vcenter{\hbox{$#2#3$ }}\kern-.6\wd0}}

\def\dashint{\Xint-}

\numberwithin{equation}{section}
\numberwithin{figure}{section}

\title{Dimension Reduction of Compressible Fluid Models \\ over Product Manifolds}

\author{Siran Li}
\address{Siran Li: Department of Mathematics, Rice University, MS 136
P.O. Box 1892, Houston, Texas, 77251-1892, USA; \, $\bullet$ \,  Department of Mathematics, McGill University, Burnside Hall, 805 Sherbrooke Street West, Montreal, Quebec, H3A 0B9, Canada;\, $\bullet$ \,  Centre de Recherches Math\'{e}matiques, Universit\'{e} de Montr\'{e}al, P.O. Box 6128, Centre-ville Station. Montr\'{e}al, Qu\'{e}bec, H3C 3J7, Canada.}

\email{\texttt{Siran.Li@rice.edu}}

\date{\today}

\begin{document}

\begin{abstract}
In this paper we study the dimension reduction limits of the compressible Navier--Stokes equations over product Riemannian manifolds $\OO_\e \cong \M \times \e\F$, such that $\dim\,(\M)=n$ and $\dim\,(\F)=d$ are arbitrary. Using the method of relative entropies, we establish the convergence of the suitable weak solutions of the Navier--Stokes equations on $\OO_\e$  to the classical solution of the limiting equations on $\M$ as $\e \map 0^+$, provided the latter exists. In addition, we also deduce the vanishing viscosity limit. The limiting equations identified through our analysis contain the weight function $A:\M \map \R^+$ as a parameter, where $A(x)$ = area of fibre $\F_x$.  Our work is based on and generalises the results in \cite{bfln} (P. Bella, E. Feireisl, M. Lewicka and A. Novotn\'{y}, 
\textit{SIAM J. Math. Anal.}, \textbf{48} (2016), 3907--3930), and it contains as special cases the physical models of circular nozzles, thin plate limits and finite-length longitudinal nozzles.

\end{abstract}

\maketitle

\section{Introduction}
	 Let $(\M, g)$ and $(\F,h)$ be Riemannian manifolds with boundaries; $\dim (\M) = n$ and $\dim (\F) =d \equiv N-n$. Throughout we assume that $\M$, $\F$ are compact and regular, {\it e.g.}, $C^{r,\alpha}$ for $r \geq 2, \alpha \in ]0,1[$. For each $\e>0$ we denote the rescaled manifold by $\e\F:=(\F, \e h)$. As a toy model for our problem, which nevertheless contains its most important features, we consider the collapse of product manifolds, {\it i.e.}, the Gromov--Hausdorff convergence:
	\begin{equation}\label{collapse of manifolds}
	\M \times \e\F \, \longrightarrow \, \M \qquad \text{ as } \e \map 0^+.
	\end{equation}
In this setting, let $u^\e$ be suitable weak solutions (see Definition \ref{def: suitable weak solution}) to the equations modelling the motion of compressible fluids on $\M \times \e\F$. We study the following question: identify the equation on $\M$ such that $u^\e$ converges, in suitable senses, to its solution in the limit \eqref{collapse of manifolds}.

The above question of {\em dimension reduction limit} is an instance of the singular limit problems of fluid models. It is a important problem in mathematical hydrodynamics, which also arises naturally in physics and engineering. In \cite{bfln}, Bella--Feireisl--Lewicka--Novotn\'{y} solved this problem for $\M = [0,1]$ and $\F = $ regular 2-dimensional flat domains embedded in $\R^3$ with varying cross-sections, which models a nozzle of finite length aligned along $z$-axis. It is proved that the suitable weak solutions of the 3-dimensional compressible Navier--Stokes and Euler equations on the nozzles $\M \times \e\F$ converge to the classical solution to the corresponding 1-dimensional equations on $\M=[0,1]$, providing the latter exists. The limiting equations contain a variable $A:\M \map \R_+$ measuring the area of cross-sections. Our work is motivated by  the aforementioned result in \cite{bfln}; we shall provide a generalisation to the product manifold with base $\M$ and fibre $\F$ of arbitrary dimensions, with fairly general geometry and topology of $\M$. In particular, we cover the case of the collapse of {\em circular nozzle} onto a circle $\sph^1 \times \e\sph^1 \map \sph^1$, for which the product manifold can be viewed as an embedded torus in $\R^3$, parametrised by 
\begin{equation*}
\mathbf{r}_\e(\theta, \phi) = \big((\e \cos\theta + R)\cos\phi, (\e\cos\theta +R)\sin\phi, \e\sin\theta \big)^\top,\qquad (\theta,\phi) \in [0,2\pi[ \times [0,2\pi[,
\end{equation*}
as well as the case of {\em thin plate limit}, {\it i.e.}, $D^2 \times \e [0,1] \map D^2$ with $D^2=\{(x,y): x^2+y^2\leq 1\}$. Let us also remark that \cite{bfln} generalises the previous convergence results  for the Navier--Stokes equations on thin rods, studied by Bella--Feireisl--Novotn\'{y} \cite{bfn}.

More precisely, our set-up of the problem is as follows. We consider a Euclidean domain $\OO \subset \R^N$, such that 
\begin{equation}
\OO \cong \M \times \F
\end{equation}
as a homeomorphism. That is, the topology of $\OO$ is that of a trivial bundle over $\M$ of fibre $\F$. We write
 \begin{equation}
\OO = \big\{(x,y): x\in \M, y \in \F_x\big\} \, \subset \, \R^{N},
\end{equation}
where $\F_x \cong \F$ is a  diffeomorphism of manifolds with $d\equiv\dim\,(\F)=N-n$, such that 
\begin{equation}
\fs := \bigcup_{x\in {\rm int}(\M)}\p\F_x \cong {\rm int}(\M) \times \p\F  \subset \p\OO
\end{equation}
is a $C^{r,\alpha}$ submanifold for some $r \geq 2$, $\alpha \in ]0,1[$. In addition, we assume
\begin{equation}\label{transversality}
\dim \,(\F_x + T_x\M) = N \qquad \text{ for each } x\in\M,
\end{equation}
{\it i.e.}, each fibre $\F_x$ is {\em transversal} to the tangent space $T_x\M$, as well as
\begin{equation}\label{injectivity rad}
{\rm inj}\,(\OO) \geq \iota_0>0,
\end{equation}
{\it i.e.}, the injectivity radius of $\OO$ is uniformly bounded away  from zero. With loss of generality ({\it cf.} Sect.\,5) we take $\F_x$ perpendicular to $T_x\M$ at every $x\in\M$, with respect to the Euclidean metric on $\R^N$. Let us also set
\begin{equation}
\fs ' := \big\{(x,y): x\in \p\M, y \in \F_x \big\} \cong \p\M \times \F \subset \p\OO;
\end{equation}
hence $\p\OO = \fs \cup \fs'$. We shall write $\nu$ for the outward unit normal vector field along $\fs$, and $\nu^\e$ its counterpart along $\fs_\e$. Furthermore, the metrics $g$ on $\M$ and $h_x$ on $\F_x$ are assumed to be compatible with the Euclidean metric $\geu$, in the sense that
\begin{equation}
g=\iota_{\M}^* \,(\geu), \qquad h_x = \iota_{\F_x}^*\,(\geu),
\end{equation}
where $\iota_\M: \M \emb \OO \emb \R^N$ and $\iota_{\F_x}:\F_x \emb \OO\emb \R^N$ are the natural inclusions. Our construction above entails that the fluid domain $\OO$, viewed as an $N$-dimensional Euclidean domain, splits as a product manifold. We also introduce the notation $\nn\in\G(T\F_x)$ as the unit normal vector field inside $\F_x$ to the {\em fibre boundary} $\p\F_x$, i.e., $\nn:\p\F_x \map \sph^{d=N-n}$ for each $x\in\M$. This shall not be confused with $\nu$, the normal to the {\em fluid boundary} $\fs$. Finally, let us denote the natural projection from $\OO$ onto $\M$ by $pr$:
\begin{equation}
\proj(y)=x \qquad \text{ whenever }\,\, y \in \F_x.
\end{equation}
In the sequel, we shall also write $pr$ for the vertical projection of curves or vector fields.

In physical terms, the boundaries of the fibres $\F_x$ are glued together nicely, so that they form a ``nozzle'' with smoothly varying $d$-dimensional cross-sections.  Our prototypical examples, including the circular nozzle and the thin plate, are special cases of the above geometric constructions. We also allow $\M$ to have non-empty boundary  $\p\M \neq \emptyset$, in order to cover the model of finite-length longitudinal nozzles. 
Moreover, one defines the rescaled fluid domain:
\begin{equation}\label{def Oe}
\OO_\e:=\big\{(x,y): x \in M, y\in \e\F_x \big\}.
\end{equation}
Then $\OO_\e$ collapses onto $\M$ as $\e \map 0^+$ in the Gromov--Hausdorff sense.

To proceed, let us fix several notations: Let $T\OO$ be the tangent bundle of $\OO$ and $\G(T\OO)$ be the space of sections of $T\OO$, {\it i.e.}, the space of vector fields on $\OO$; we also write $\G^k(T\OO)$ for the space of $C^k$ vector fields. Then, one can define globally and intrinsically the gradient $\na: C^{k+1}(\OO) \map \G^k(T\OO)$, as well the divergence $div \equiv {\rm trace}\,\na : \G^{k+1}(T\OO) \map \R$, for $k \geq 0$. For a vector field $u \in \G(T\OO)$ with suitable regularity, one can also consider its gradient $\na u\in \G(T^*\OO \otimes T\OO)=\G(\gl(T\OO))$, where the tensor product $T^*\OO \otimes T\OO$ is identified with $\gl(T\OO)$, the space of linear transforms on $T\OO$, namely the space of $N \times N$ matrices. To be more precise, we use $\na$ to denote the covariant derivatives on $\OO$, which maps differential $r$-forms to $(r+1)$-forms (or the associated contra-variant tensor fields via contraction). The divergence of a symmetric tensor field $S=\{S^{ij}\}$ is defined by $\{\dv (S)^j\}_1^N := \na_i S^{ij}$; here and throughout the Einstein summation convention is adopted.  Moreover, one introduces the {\em deformation tensor} 
\begin{equation}
\defo  :=\frac{\na u + \na^\top u}{2} = \bigg\{\frac{\na_i u^j + \na_j u^i}{2}\bigg\}_{1\leq i,j \leq N},
\end{equation}
where $\top$ denotes the transpose of a matrix, and $u=(u^i)_{i=1}^N$ in local coordinates. Then, the {\em stress tensor} $\sph(\na u)\in\G(T^*\OO \otimes T\OO)$ is given by
\begin{equation}\label{stress tensor definition}
\sph(\na u) = \mu\Big(2\defo - \frac{2}{N}\,\dv (u)\,\id \Big) + \eta\, \dv (u) \,\id,
\end{equation}
where $\id$ is the $N \times N$ identity matrix and $\mu >0$, $\eta >(2/N-2)\mu$ are the the shear and bulk viscosity constants specific to the fluid. We note that $\defo, \sph(\na u)$ are both symmetric tensor fields. The preceding constructions extend naturally to $\OO_\e$ and $\fs_\e$.

With the above preparation, we are at the stage of formulating the compressible Navier--Stokes equations on $\OO$. Consider the following Cauchy problem of the PDE system in terms of the density and velocity of the fluid $(\rho,u): [0,T]\times \OO \map \R_+ \times T\OO$:
\begin{eqnarray}
&& \p_t \rho + \dv\, (\rho u) = 0 \qquad \text{ in } [0,T] \times \OO, \label{continuity eq}\\
&& \p_t(\rho u) + \dv\,(\rho u \otimes u) + \na p(\rho)-  \dv\,\big(\sph(\na u)\big) = 0 \qquad \text{ in } [0,T] \times \OO,\\ \label{momentum eq on O} 
&& u \cdot \nu = 0, \quad [\sph(\na u) \cdot \nu] \times \nu = 0 \qquad \text{ on } [0,T] \times \fs,\\ \label{slip bc on O}
&& u \equiv 0 \qquad \text{ on } [0,T] \times \fs',\\\label{bc on partial M}
&& (\rho, u)|_{t=0} = (\rho_0, u_0)  \qquad \text{ on } \{0\} \times \OO.\label{initial data}
\end{eqnarray}
Eq.\,\eqref{continuity eq} is the {\em continuity equation} accounting for the conservation of mass, and Eq.\,\eqref{momentum eq on O} is the conservation law for the momentum. In Eq.\,\eqref{momentum eq on O}, the stress tensor $\sph(\na u)$ is given by Eq.\,\eqref{stress tensor definition} with fixed constants $\mu, \eta$; in addition, $p=p(\rho)\in C([0,\infty[) \cap C^3(]0,\infty[)$ is the pressure of the fluid, such that $p(0)=0$, $p'>0$ on $]0,\infty[$. Eq.\,\eqref{slip bc on O} is known as the (complete) {\em slip boundary condition}, and Eq.\,\eqref{bc on partial M} is the Dirichlet condition on the fibres over $\p\M$. To formulate the rescaled compressible Navier--Stokes equations on $\OO_\e$, we simply take $(\rho^\e,u^\e): [0,T]\times \OO_\e \map \R_+ \times T\OO_\e$ that satisfies
\begin{eqnarray}
&& \p_t \rho^\e + \dv\, (\rho^\e u^\e) = 0 \qquad \text{ in } [0,T] \times \OO_\e, \label{continuity eq on Oe}\\
&& \p_t(\rho^\e u^\e) + \dv\,(\rho^\e u^\e \otimes u^\e) + \na p(\rho^\e)-  \dv\,\big(\sph(\na u^\e)\big) = 0 \qquad \text{ in } [0,T] \times \OO_\e \label{momentum eq on Oe},\\
&& u^\e \cdot \nu^\e = 0, \quad [\sph(\na u^\e) \cdot \nu^\e] \times \nu^\e = 0 \qquad \text{ on } [0,T] \times \fs_\e,\\ \label{slip bc on Oe}
&& u^\e \equiv 0 \qquad \text{ on } [0,T] \times \fs_\e',\\\label{bc on partial Me}
&& (\rho^\e, u^{\e})|_{t=0} = (\rho_0, u_0) \qquad \text{ on } \{0\} \times \OO_\e.\label{initial data for Oe}
\end{eqnarray}
The problem {\bf (P--NS)} of the dimension reduction limit of compressible Navier--Stokes equations asks about the convergence of the solutions to Eqs.\,\eqref{continuity eq on Oe}--\eqref{initial data for Oe} as $\e \map 0^+$. 

In addition, we shall also investigate the limit as {\em both} $\e\map 0^+$ {\em and} $\nu,\eta \map 0^+$ in Eqs.\,\eqref{continuity eq on Oe}--\eqref{initial data for Oe}. In other words, we consider simultaneously the dimension reduction limit and the vanishing viscosity limit. One natural conjectures that the weak solutions should converge to those of (a variant of) the Euler equations on $\M$, taken into account the geometrical effects of the non-uniform fibres $\F_x$. This problem is denoted by {\bf (P--Euler)}.  

Before further developments, let us remark that the geometric formulation of the Navier--Stokes equations on Riemannian manifolds has been a well-developed topic in global analysis and mathematical hydrodynamics, though mainly for incompressible fluids; {\it cf.} the pioneering works by Arnol'd \cite{arnold}, Ebin--Marsden \cite{em} and Shnirelman \cite{s}.  On the other hand, the weak solutions to the Euler and Navier--Stokes equations in the longitudinal nozzles with varying cross-sections have been studied intensively; see Chen--Glimm \cite{cg}, LeFloch--Westdickenberg \cite{lw} and the references cited therein. In the nozzle problems, the geometrical effects caused by the curvilinear fluid boundaries are crucial to the mathematical analysis.

Now let us describe the limiting equations on $\M$. For each $x\in\M$, the fibre $\F_x\cong\F$ is a $d$-dimensional submanifold of $\R^N$, so its $d$-dimensional Hausdorff measure is well defined. Let us denote by $A: \M \map \R_+$, where
\begin{equation}
A(x):=\haus^d (\F_x) \qquad \text{ for each } x\in\M.
\end{equation}
Moreover, we use the symbols $\hr, \hu$ and $\hna,\hD, \hdv$... to denote the variables and operators on $\M$. The limiting equations on $\M$ for the Navier--Stokes Eqs.\,\eqref{continuity eq on Oe}--\eqref{initial data for Oe} are the following equations in $(\hr, \hu):[0,T]\times \M \map \R_+ \times T\M$:
\begin{eqnarray}
&& \p_t (\hr A) + \hna (\hr\hu A) = 0 \qquad \text{ in } [0,T] \times \M, \label{limiting continuity eq, NS}\\
&& \hr\, (\p_t\hu + \hu \cdot \hna \hu) + \hna p(\hr) \nonumber\\
&&\qquad\qquad\qquad= \mu \hD\hu + \big(\eta + \frac{N-2}{N}\mu\big)\hna\hdv\hu + \TT(A,\hu) \qquad 
\text{ in } [0,T] \times \M,\label{limiting momentum eq, NS}\\
&& \hu = 0 \qquad \text{ on } [0,T] \times \p\M,\label{limiting bc on pM, NS}\\
&& (\hr, \hu)|_{t=0} = (\widehat{\rho_0}, \widehat{u_0}) \qquad \text{ on } \{0\}\times \M. \label{limiting initial data, NS}
\end{eqnarray}
Here, $\TT(A,\hu)$ is the only term in the limiting momentum equation \eqref{limiting momentum eq, NS} containing $A$:
\begin{equation}\label{T term}
\TT(A,\hu) := \vis \hna\Big(\hna_{\hu}\log A\Big).
\end{equation}
It is proportional to the Hessian  $\hna\hna\log A$ and accounts for the geometrical effects, {\it i.e.}, the variation of areas of cross-sections along $\M$. In particular, if $A \equiv {const.}$ on $\M$, as in the cases of a straight cylindrical nozzle ($\M=[0,1]$, $\F_x=D^2$ for all $x$) or a circular nozzle with fixed cross-section ($\M=\sph^1$, $\F_x=D^2$ for all $x$), then $\TT(A,\hu)\equiv 0$. Roughly speaking, Eq.\,\eqref{T term} suggests that the geometrical effects in the dimension reduction limit is manifested in the viscous terms.

To describe the limiting equations for the Euler system, we may simply drop the second-order terms and the $\TT$ term in Eq.\,\eqref{limiting momentum eq, NS}:
\begin{eqnarray}
&& \p_t (\hr A) + \hna (\hr\hu A) = 0 \qquad \text{ in } [0,T] \times \M, \label{limiting continuity eq, EULER}\\
&& \hr\, (\p_t\hu + \hu \cdot \hna \hu) + \hna p(\hr) =  0 \, 
\text{ in } [0,T] \times \M,\label{limiting momentum eq, EULER}\\
&& \hu = 0 \qquad \text{ on } [0,T] \times \p\M,\label{limiting bc on pM, EULER}\\
&& (\hr, \hu)|_{t=0} = (\widehat{\rho_0}, \widehat{u_0}) \qquad \text{ on } \{0\}\times \M. \label{limiting initial data, EULER}
\end{eqnarray}
Let us emphasise that Eq.\,\eqref{limiting momentum eq, EULER}, the limiting momentum equation for Euler, is independent of the area parameter $A$. The primary goal of this paper is to rigorously derive Eqs.\,\eqref{limiting continuity eq, NS}--\eqref{limiting bc on pM, NS} and Eqs.\,\eqref{limiting continuity eq, EULER}--\eqref{limiting bc on pM, EULER} as the limits of Eqs.\,\eqref{continuity eq on Oe}--\eqref{initial data for Oe},  as $\e\map 0^+$ and $\e,\eta,\mu\map 0^+$, respectively.

In brief, the main results of our paper provide a solution to the Problems {\bf (P--NS)} and {\bf (P--Euler)} in the affirmative.   Assume that $(\hr, \hu)$ is a classical solution to Eqs.\,\eqref{limiting continuity eq, NS}--\eqref{limiting initial data, NS}. Then, any {\em suitable weak solution} $(\rho^\e, u^\e)$ to the Navier--Stokes system \eqref{continuity eq on Oe}--\eqref{initial data for Oe} converges to $(\hr, \hu)$ as $\e\map 0^+$, in a sense suitably described by {\em relative entropies}. 
Moreover, as the viscosity coefficients $\mu, \eta \map 0^+$ additionally, one can also establish vanishing viscosity limit of the suitable weak solutions to the classical solution to Eqs.\,\eqref{limiting continuity eq, EULER}--\eqref{limiting initial data, EULER}.  The precise statement of these results are in Sect.\,3.

 Our work is closely related to the theory of ``weak-strong uniqueness'' in the PDEs modelling fluid dynamics and continuum mechanics; see Dafermos \cite{d2}, Brenier--De Lellis--Sz\'{e}kelyhidi \cite{bds}, Germain \cite{germain}, Wiedemann \cite{w} and the many references cited therein. Also, let us emphasise once more that our work is based on, as well as extends, the main results in \cite{bfln} (P. Bella, E. Feireisl, M. Lewicka and A. Novotn\'{y}, 
\textit{SIAM J. Math. Anal.}, \textbf{48} (2016), 3907--3930). In particular, we obtain the generalisation to the dimension reduction (and vanishing viscosity) limits in arbitrary dimensions and co-dimensions. The limiting equations for the Navier--Stokes system --- Eqs.\,\eqref{limiting continuity eq, NS}--\eqref{limiting initial data, NS} ---  take a more complicated form than the $1$-dimensional case in \cite{bfln}.

\smallskip 
The remaining parts of the paper are organised as follows: In Sect.\,2 we discuss several geometric properties of our problem. In particular, we describe a canonical way of lifting any $\hx \in \G(T\M)$ to vector fields on $\OO$ and $\fs$. Next, in Sect.\,3 we review the definition of relative entropy and suitable weak solutions to the compressible Navier--Stokes equations.  In Sect.\,4, by selecting appropriate test functions (based on the canonical lifting in Sect.\,2) for the relative entropy (introduced in Sect.\,3), we prove the convergence from suitable weak solutions to the Navier--Stokes equations in the dimension reduction limit and, in addition, the vanishing viscosity limit. Finally, in Sect.\,5, we briefly remark on several problems for future study.

\section{Canonical lifting of horizontal vector fields}

In this section, we make a simple geometric observation that shall play a crucial role in the future developments: Given a vector field $\hx \in \G(T\M)$, a ``canonical lift'' of $\hx$ to the fluid domain $\OO$ can be constructed; see Definition \ref{def: canonical lift in O} below. Sect.\,2 generalises the calculations on pp.\,3909--3911 in Bella--Feireisl--Lewicka--Novotn\'{y} \cite{bfln}.

Recall that in the Introduction (Sect.\,1), the topology of the fluid domains is prescribed: $\OO_\e \cong\OO \cong \M \times \F$. Thus, one can lift any curve $\hg\subset{\rm int}\,(\M) \times \{0\}$ vertically, thanks to the triviality of the fibre bundle $\OO$. Throughout the convention is to view $T\M$ as {\em horizontal} and the fibre $\F$ as {\em vertical}. Such lifting preserves the transversality condition \eqref{transversality}: denoting the lifting map by $\lift$, namely
\begin{equation}\label{lift}
[\lift\hg]_y \equiv \lift\hg_y = \gamma_y \qquad \text{ for } y \in \F_x,
\end{equation}
we have
\begin{equation}\label{transversality of lift}
\lift \hg_y \text{ is transversal to } \F_x \text{ for any } y\in\F_x.
\end{equation}
Alternatively, we can also view $\lift$ as an map from $\G(T\M)$ to $\G(T\OO)$. Let $\hg: ]-\delta,\delta[\map\M$ be a $C^1$ curve such that $\hg(0)=x$, $\dot{\hg}(0)=\hx_x$; then one sets
\begin{equation}
\lift \hx_y := \frac{d}{dt}\big|_{t=0} (\lift\hg_y)(t) \qquad \text{ for each } y \in \F_x.
\end{equation}
This provides the vertical lifting for  vector fields.

We note that
\begin{equation}\label{lift dot n nonzero}
\langle\lift \hx_y, \nn_x \rangle \neq 0 \qquad \text{ for any }y \in \fs \cap \F_x \text{ and }  \hx \in \G(T\M)
\end{equation}
where $\langle\cdot,\cdot\rangle$ is the Euclidean inner product, 
due to the transversality of the lifting \eqref{transversality of lift}. The fluid domain $\OO$ is foliated by diffeomorphic copies of $\M$. Again by \eqref{transversality of lift}, one deduces
\begin{equation}\label{n dot nu nonzero}
\langle\nn_y, \nu_y\rangle \neq 0\qquad \text{ for any } y\in\fs.
\end{equation}

\begin{lemma}\label{lemma: geometry}
There exists a nowhere vanishing map $C^{r-1,\alpha}$ map $\beta: T\M \times \fs \map \R$ such that
\begin{equation*}
\Big\langle\beta(\hx,y)\nn_y + \lift\hx_y, \nu_y\Big\rangle = 0 \qquad \text{ for any } y \in \F_x\cap\fs, x\in\M \text{ and } \hx \in \G(T\M).
\end{equation*}
\end{lemma}
\begin{proof}
	With $\hx\in\G(T\M)$ and $y\in\fs$ fixed, let us set $	\beta_y(\hx):= {-\langle\lift\hx_y, \nu_y\rangle}/\langle\nn_y, \nu_y\rangle$. By \eqref{n dot nu nonzero} $\beta_y$ is well defined, and by \eqref{lift dot n nonzero} it is nowhere vanishing.  Since $\fs$ is a $C^{r,\alpha}$ hypersurface in $\R^N$ and $\M$ is a smooth manifold, $\nn,\nu$ and $\lift\hx$ are at least in $C^{r-1,\alpha}$. The proof is complete.  \end{proof}

\begin{definition}\label{def: canonical lift}
Let $\OO \cong \M \times \F$ be as before, and let $\hx \in \G(T\M)$.   The canonical boundary lift of $\hx$ is the $C^{r-1,\alpha}$ vector field $\clx$ in $\G(T\fs)$, given by
\begin{equation}
\clx_y :=\beta(\hx, y)\nn_y + \lift\hx_y \qquad \text{ for each } y\in\fs.
\end{equation}
\end{definition}

In other words, there is a canonical way to lift any vector field tangential to $\M$ to a vector field tangential to the fluid boundary $\fs$, which transverses the fibres at the same speed of the tangential field $\hx$. The existence of $\clx$ is guaranteed by Lemma \ref{lemma: geometry}.

Now, fix $x\in\M$ and a curve $\hg \subset {\rm int}\,(\M)$. With a slight abuse of notations, we also denote by $\clx$ its own extension to $\overline{\OO}$, whose existence is ensured by $\langle\clx,\nu\rangle\equiv 0$. The arguments in Lemma \ref{lemma: geometry} show that $\clx$  is transversal to each fibre $\F_x$. Given a curve $\hg \subset \M$ such that $\hg(0)=x_0$, let us consider $\hx:=\dd \hg/\dd t$, the corresponding vector field along $\hg$; as before, $\gamma=\lift\hg$ stands for the lifted curve in $\OO$. If we set $\phi:]-\delta,\delta[\times({\rm im}\,(\gamma)\subset\OO)\map\OO$ to be the flow of $\clx$, determined by the ODE:
\begin{equation}\label{ODE}
\begin{cases}
\frac{\dd}{\dd t}\phi(t,y) = \clx\big(t,\phi(t,y)\big)\\
\phi|_{t=0}=\id_{\F_{x_0}},
\end{cases}
\end{equation}
then the area $A(x)$ of each fibre $\F_x$ can be computed as follows:
\begin{equation}
A(x):= \int_{\F_{\gamma(0)}} \bigg\{ \det \na \phi(t,y) \bigg\}\,\dd \haus^d(y) \qquad \text{ where } x=\gamma(t).
\end{equation}
Using Jacobi's formula for the derivative of determinant, ODE \eqref{ODE}, the change of variables formula and Stokes' theorem, we obtain 
\begin{align}
\frac{\dd}{\dd t}A(\gamma(t)) &= \int_{\F_{\gamma(0)}} \bigg\{ \det \na\phi(t,y)\,\, {\rm trace}\,\Big\{\,\Big(\na\phi(t,y)\Big)^{-1} \cdot \Big(\frac{\dd}{\dd t} \na\phi(t,y)\Big) \Big\} \bigg\}\dd\haus^d(y)\nonumber\\
&= \int_{\F_{\gamma(0)}}\bigg\{ \det \na\phi(t,y)\,\, {\rm trace}\,\Big\{\,\Big(\na\phi(t,y)\Big)^{-1} \cdot  \na  \clx\big(t,\phi(t,y)\big)     \Big\} \bigg\}\dd\haus^d(y)\nonumber\\
&=\int_{\F_{\gamma(t)}} \dv\,\clx(t,y)\,\dd\haus^d(y)\nonumber\\
&= \int_{\p\F_x} \clx (t,z) \cdot \nn(z)\,\dd \haus^{d-1}(z).
\end{align}
Therefore, in view of the local nature of the directional derivative, we have proved:
\begin{lemma}\label{lemma: evolution of fibre area}
Let $\hna$ be the covariant derivative on $\M$ and let $\hx\in\G(T\M)$. Then the derivative of $A$ $($the area of fibres$)$ can be computed as follows:
\begin{equation}
\hna_{\hx} A (x) = \int_{\p\F_x} \clx \cdot \nn \, \dd\haus^{d-1} \qquad\text{for each } x\in\M. 
\end{equation}
Here $\clx$ is the canonical boundary lift of $\hx$ as in Definition \ref{def: canonical lift}.
\end{lemma}


We need to further extend $\clx$ from the fluid boundary $\fs$ to $\OO$. To this end, we shall first define a vertical extension of $\clx$ in each $\F_x$, which in addition has horizontal regularity across the fibres. This is achieved by considering a boundary value problem in each fibre:

\begin{lemma} \label{lemma: dlogA}
Let $A$ be as before. There exists a tensor field $V \in \G(T^*\M\otimes T\F)$, such that $V(x, \hx)\equiv V_{\hx}(x) \in \G(T\F_x)$ for each $x\in \M, \hx \in \G(T\M)$; $V_{\hx}(x)|_{\fs}=\proj\clx(x)|_{\fs}$; and  that
\begin{equation}
\hna_{\hx} \, \big\{\log A(x)\big\} = \dv_\F V_{\hx}(x) \qquad \text{ for each } x\in\M.
\end{equation}
Here, $\dv_\F$ on the right-hand side is the divergence on $\F_x$.
\end{lemma}

\begin{proof}
	Fix $x\in\M$ and $\hx\in\G(T\M)$. Let us consider the following Neumann problem of the Poisson equation for $U_{\hx}:\F_x\map\R$, which depends on $\hx$:
	\begin{equation}\label{poisson eq}
\begin{cases}
\Delta_\F U_{\hx} = \hna_{\hx} \, \big\{\log A(x)\big\} \qquad \text{ in } \F_x,\\
\na_\F U_{\hx}\cdot \nn = \langle\clx, \nn\rangle \qquad \text{ on } \p\F_x.
\end{cases}
	\end{equation}
Here and throughout, the differential operators $\Delta_\F, \na_\F$ denote the Laplacian and the gradient in $\F_x$. As $A(x)$ and $\hna_{\hx}A$ are constant on $\F_x$, by Lemma \ref{lemma: evolution of fibre area} one has
\begin{equation}\label{U}
\int_{\F_x} \hna_{\hx} \log A(x) \,\dd\haus^{d} = [A(x)]^{-1} \int_{\F_x} \hna_{\hx} A(x)\,\dd\haus^d = \hna_{\hx} A(x)= \int_{\p\F_x} \langle\clx, \nn\rangle \, \dd\haus^{d-1}.
\end{equation}
Thus, there exists a solution in the fibre $U_{\hx} \in C^{r-1,\alpha}(\F_x)$ for each fixed $x$ ({\it e.g.}, as a variant of the theorem in Chapter 4, Weirheim \cite{uhlenbeck}). Let us set 
\begin{equation*}
V_{\hx}:=\na_\F U_{\hx}
\end{equation*}
where $U_{\hx}$ is determined by $\hx$ from Eq.\,\eqref{U}. By Eq.\,\eqref{poisson eq} above, we thus have $\dv_\F \,(V_{\hx}) = \Delta_\F U_{\hx} = \hna_{\hx}\log A$ in each $\F_x$.

In the above we have constructed $V_{\hx}$ in each fibre $\F_x$ with the $C^{r-1,\alpha}$ regularity. It remains to show that $V_{\hx}$ has the same regularity in the entire $\OO$, i.e., ``regular across the fibres''. To this end, let us restrict the arbitrary vector field $\hx \in \G(T\M)$ to the smooth curve $\hg:]-\delta,\delta[ \map \M$ passing through $\hg(0)=x_0$. As before, we also write $\clx$ for the canonical lift of $\hx|_{{\rm im}\,(\hg)}$, which is a vector field defined along the lifted curve $\lift\hg=:\gamma$, such that $\proj (\gamma(0))=x_0$ for the natural projection $\proj:\OO\map\M$. As in \eqref{ODE} we denote by $\phi(t,\cdot)\equiv\phi_t(\cdot)$ the flow of $\clx$.

Now, let us pull back Eq.\,\eqref{ODE} by $\phi_t$ to $\F_{x_0}$: 
\begin{equation}
\begin{cases}
\dv_\F\,\Big( \big\{\det\na_\F\phi_t\big\}\big\{\na_\F\phi_t \big\}^{-2} \na_\F [\phi_t^*U]\Big) = (\det \na_\F\phi_t)\, \frac{\dd}{\dd t}\log A\qquad \text{ in } \F_{x_0},\\
\na_\F (\phi_t^*U)\cdot (\phi_t^*\nn) = \langle\phi_t^*\clx, \phi_t^*\nn\rangle \qquad \text{ on } \p\F_{x_0}.
\end{cases}
\end{equation}
Thus, taking another directional derivative $\dd/\dd t=\hna_{\hx}$, we obtain a second-order nonlinear elliptic equation in $\dd U/\dd t$ with continuous coefficients, and similarly for the higher order horizontal derivatives. By the standard elliptic estimates we obtain the H\"{o}lder regularity in the horizontal, namely $\G(T\M)$, direction along $\hx$. Since $\hx$ is arbitrary, the proof is complete.  \end{proof}

The above lemma justifies the following
\begin{definition}\label{def: canonical lift in O}
Let $\hx \in \G(T\M)$ be a tangential vector field. Its canonical lift is a tangential vector field in the fluid domain $X\in \G(T\OO)$, defined by
\begin{equation}
X := (\id \oplus V)(\hx) = (\hx,V_{\hx})^\top  \in \G(TM \oplus T\F) \cong \G(T\OO).
\end{equation}
\end{definition}
By construction, the horizontal projection of $X$ is $\hx$, and the vertical projection on $X$ restricted to the fluid boundary $\fs$ coincides with the canonical boundary lift $\clx$ on each fibre $\F_x$. In geometric terms, $X$ yields a canonical choice of the trivial connection on the bundle $\overline{\OO} \cong \M\times\overline{\F}$ by specifying a horizontal section.

Now we are at the stage of proving the following important result: 
\begin{proposition}\label{propn: conservation law}
Let $A(x):=\haus^d(\F_x)$ be as above. Let $(\hr, \hu):[0,T]\times \M \map \R_+ \times \G(T\M)$ satisfy the following ``weighted continuity equation'':
\begin{equation}\label{weighted continuity equation}
\p_t\,(\hr A) + \dv (\hr\,\hu A) = 0.
\end{equation} 
Then there holds 
\begin{equation}\label{associated conservation law}
\p_t \hr + \dv (\hr\, U) = 0,
\end{equation}
where $U\in\G(T\OO)$ is the canonical lift of $\hu$.
\end{proposition}
\begin{proof}
 Let $U=(\hu, U_\F)$ be the decomposition into horizontal and vertical directions. We have
	\begin{align}
\p_t \hr + \hdv (\hr\, U) &= \frac{1}{A}\bigg\{ \p_t(\hr A) + \dv (\hr\, UA) - \hr\, U \cdot\na A \bigg\}\nonumber\\
&=\frac{1}{A}\bigg\{ \p_t(\hr A) + \hdv (\hr\, \hu A) +\dv_{\F}(\hr\,U_\F A) - \hr\, \hu \cdot\hna A - \hr \, U_\F\cdot\na_\F A \bigg\}.
	\end{align}
The first two terms in the bracket add up to zero by Eq.\,\eqref{weighted continuity equation}; in addition, $\na_\F A=0$ as $A$ is constant on each fibre. Hence,
\begin{equation}
\p_t \hr + \hdv (\hr\, U) = -\frac{\hr}{A} \Big\{\hu\cdot\hna A - (\dv_\F U)A \Big\},
\end{equation}
which is equal to zero in light of Lemma \eqref{lemma: dlogA}. The proof is complete.  \end{proof}

Recall that $\OO_\e$ is obtained from $\OO$ by rescaling in the vertical direction; see Eq.\,\eqref{def Oe}. We also define, for $\e\in]0,1[$, 
\begin{equation}\label{def: Ue}
X_\e := (\id \oplus V_\e)(\hx) = (\hx, V_{\e,\hx})^\top \in \G(T\M \oplus T[\e\F])\cong \G(T\OO_\e),
\end{equation}
where $V_\e \in \G(T^*\M \otimes T[\e\F])$ is given by
\begin{equation}\label{def of V epsilon}
V_{\e,\hx}(y) = \e V_{\hx}\big(\frac{y}{\e}\big) \qquad \text{ for each } y \in \F_x, \,\, x\in\M \text{ and } \hx \in \G(T\M).
\end{equation}
If $X\in\G(T\OO)$ is the canonical lift of some $\hx\in \G(T\M)$, then clearly $X_\e \in \G(T\OO_\e)$ is the canonical lift of $\hx$ to $\OO_\e$. Thus, Proposition \ref{propn: conservation law} remains valid when $X$ ($\OO$) is replaced by $X_\e$ ($\OO_\e$, resp.) therein. 


\section{Relative entropy and suitable weak solutions}

In this section we discuss the relative entropy functional and the weak formulation of the compressible Navier--Stokes Equations \eqref{continuity eq} -- \eqref{initial data}.

In \cite{germain} Germain introduced a class of weak solutions to the compressible Navier--Stokes equations satisfying the relative entropy inequality, in order to prove the weak-strong uniqueness results in this class. Then, in the spirit of \cite{germain}, Feireisl--Novotn\'{y}--Sun defined intrinsically the notion of {\em suitable weak solutions} and established the gloabl existence within this class for any finite-energy initial data, via an approximation scheme. Later in \cite{fjn} Feireisl--Jin--Novotn\'{y} proved that, in effect, every finite-energy weak solution is a suitable weak solution. In this paper we adopt the definition of the suitable weak solutions from \cite{fjn, fns}, but consider on the $N$-dimensional  domain $\OO$ as above, rather than the $3$-dimensional domains.

For this purpose, let us first introduce the {\em relative entropy functional} $\E$. Let $(\rho,u)$ and $(r,U)$ be two pairs of functions mapping $[0,T]\times\OO$ to $\R\times T\OO$. Consider the renormalisation function $H:\R_+\map\R$ defined, for some fixed constant $\lrho \geq 0$, as follows:
\begin{equation}
H(\rho):=\rho \int_{\lrho}^\rho \frac{p(\sigma)}{\sigma^2}\,\dd\sigma.
\end{equation} 
First, the function $H$ is defined as the (formal) solution to the following ODE:
\begin{equation}\label{def of H}
r H'(r) - H(r) = p(r) \qquad \text{ for } r>0.
\end{equation}
Second, the second derivative of $H$ takes a simple form:
\begin{equation}\label{H''}
H''(r)=\frac{p'(r)}{r}\qquad \text{ for } r>0.
\end{equation}
Then, setting
\begin{equation}\label{def of E}
\E(\rho,u|r,U) := \int_{\OO} \bigg\{\frac{1}{2}\rho|u-U|^2 + \Big[H(\rho)-H'(r)(\rho-r) - H(r)\Big] \bigg\}\,\dd x,
\end{equation}
we adopt the following notion (Definition \ref{def: suitable weak solution}) of suitable weak solutions from \cite{fns}.  
Throughout, for $N\times N$ matrices $A$ and $B$ we write $A:B = \sum_{1\leq i,j\leq N}A^i_jB^j_i$; for vectors $v\in \R^{N_1}$ and $w\in \R^{N_2}$, we denote by $v \otimes w$ the $N_1 \times N_2$ matrix $\{v^iw_j\}_{1\leq i \leq N_1, \, 1\leq j \leq N_2}$. Note that the analogous constructions can be done on $\OO_\e$, for each $\e>0$, to Eqs.\,\eqref{continuity eq on Oe} -- \eqref{initial data for Oe}. 
\begin{definition}\label{def: suitable weak solution}
Let $\OO\subset\R^N$ be the fluid domain as above. $(\rho, u): [0,T]\times \OO\map\R_+\times T\OO$ is a suitable weak solution to the compressible Navier--Stokes Equations \eqref{continuity eq} -- \eqref{initial data} if the following hold:
\begin{enumerate}
\item
$(\rho,u)$  satisfies Eqs.\,\eqref{continuity eq} -- \eqref{initial data} in the sense of distributions;
\item
For an arbitrary smooth, strictly positive function $r:[0,T]\times\OO\map\R_+$ and an arbitrary smooth vector field $U:[0,T]\times\OO\map T\OO$ with the boundary conditions \eqref{slip bc on O}, the {\em relative entropy inequality} holds as follows:
\begin{align}\label{relative entropy}
&\E(\rho,u|r,U) (t) + \int_0^t\int_{\OO} \big[\sph(\na u)-\sph(\na U)\big]:(\na u - \na U)\,\dd x\dd t'\nonumber\\
&\qquad\qquad\qquad \leq  \E(\rho,u|r,U) (0) +\int_0^t \RR(\rho,u,r,U)(t') \,\dd t',
\end{align}
for almost every $t \in [0,T]$. Here the remainder term $\RR$ is given by
\begin{align}\label{def of R}
\RR(\rho,u,r,U) &:= \int_{\OO} \rho\,\p_t U \cdot (U-u)\,\dd x + \int_{\OO} \rho\,\na U : \Big[u \otimes (U-u)\Big]\,\dd x \nonumber\\
&\qquad - \int_{\OO} \sph(\na U):(\na u - \na U)\,\dd x + \int_{\OO} \dv\, U \Big[p(r)-p(\rho)\Big]\,\dd x\nonumber\\
&\qquad +\int_{\OO} \bigg\{(r-\rho) \p_t H'(r) + \big(rU-\rho u\big)\cdot \na H'(r)\bigg\}\,\dd x.
\end{align}
\end{enumerate}
\end{definition}

Let us impose the following assumptions on the pressure function $p=p(\rho)$ as in \cite{bfln}:
\begin{equation}\label{assumption on p}
\begin{cases}
p \in C^0([0,\infty[) \cap C^3 (]0,\infty[),\\
p(0)=0,\\
p'>0 \text{ on } ]0,\infty[,\\
\lim_{\rho \map \infty} \big(\rho^{1-\gamma}\,p'(\rho)\big) =: p_\infty > 0 \quad \text{ for a certain } \gamma > N/2. 
\end{cases}
\end{equation}
These assumptions guarantee the existence of suitable weak solutions to the compressible Navier--Stokes Eqs.\,\eqref{continuity eq} -- \eqref{initial data} on $\OO$ (and hence on $\OO_\e$, by scaling); see  Lions \cite{lions}, Feireisl \cite{f} and the discussions below.     Moreover, they ensure the following useful identities (Eq.\,(2.10) in \cite{bfln}) for the integrand of the relative entropy:
\begin{align}\label{p1}
C_1(K) &\Big(|u-U|^2+|\rho-r|^2 \Big)\nonumber\\
\leq\,& \frac{1}{2}\rho|u-U|^2 + H(\rho) - H(r) - H'(r)(\rho-r) \leq C_2(K) \Big(|u-U|^2+|\rho-r|^2 \Big)
\end{align}
for all $\rho, r \in K \Subset ]0,\infty[$ compact, as well as
\begin{align}\label{p2}
\frac{1}{2}\rho|u-U|^2 + H(\rho) - H(r) - H'(r)(\rho-r) \geq C_3(K,K')\Big(1+\rho|u-U|^2 + \rho^\gamma\Big)
\end{align}
for all $r \in K \subset {\rm int}\,(K')$ and $\rho \in [0,\infty[\setminus K'$, where $K' \Subset ]0,\infty[$ is compact.

In passing, let us remark that the ``dissipative weak solutions'' in the sense of Lions \cite{lions} are, in fact, suitable weak solutions. The $N=3$ case is proved by Feireisl--Novotn\'{y}--Petzeltov\'{a} \cite{fnp} and Feireisl--Jin--Novotn\'{y} \cite{fjn}. The proof therein carried over the case of arbitrary $N$.

\section{Convergence of solutions as the product manifolds collapse}

\subsection{Dimension reduction limit of the Navier--Stokes system}

In this subsection we prove the first main result of the paper, Theorem \ref{thm: convergence of NS}. It answers the problem {\bf (P--NS)} in the affirmative, {\em provided that the classical solution exists}. More precisely, it establishes the convergence from {\em suitable weak solutions} of the compressible Navier--Stokes equations on $\OO_\e$ to the {\em classical solution} of the limiting equations on $\M$, as the product manifolds $\OO_\e$ collapse to $\M$. 
\begin{theorem}\label{thm: convergence of NS}
Let $\OO \cong \M \times \F \subset \R^N$ be a fluid domain as in Sect.\,$1$. Let $(\rho^\e,u^\e):[0,T]\times \OO_\e \map \R_+ \times T\OO_\e$ be a family of suitable weak solutions to the Navier--Stokes equations \eqref{continuity eq on Oe}--\eqref{initial data for Oe}, indexed by $\e \map 0^+$, whose pressure term $p$ satisfies \eqref{assumption on p}. Suppose that Eqs.\,\eqref{limiting continuity eq, NS}--\eqref{limiting initial data, NS} have a classical solution $(\hr, \hu):[0,T]\times \M \map \R_+ \times T\M$, which satisfies
	\begin{equation}
	\Lambda:={\rm ess \, sup}_{t\in[0,T]}\, \big\|\hu(t,\cdot)\big\|_{C^2(\M)} < \infty,
	\end{equation}
as well as
\begin{equation}\label{assumption on rho hat}
0<\lrho \leq \hr(t,x) \leq \urho <\infty\qquad \text{ on } [0,T]\times\M
\end{equation}
for some constants $\lrho,\urho$.  The bulk viscosity constant $\eta$ is strictly positive. Moreover, assume that the rescaled domains $\OO_\e$ support the uniform Korn's  inequality as $\e\map 0^+$, in the following sense: there exists a constant $C>0$, independent of $\e$, such that
\begin{equation}\label{uniform korn ineq}
\|\phi\|_{W^{1,2}(\OO_\e)} \leq C \|\mathbb{D}(\na \phi)\|_{L^2(\OO_\e)}\qquad \text{ for any } \phi \in W^{1,2}(\OO_\e) \text{ for each } \e>0.
\end{equation}
Then, the following holds: There exists a constant $C_0$, depending only on $\Lambda$, $T$, the geometry of $\OO$ and the physical  constants specific to the fluid, such that for {\it a.e.} $t\in ]0,T]$, we have
\begin{equation}
\frac{1}{|\OO_\e|} \E^\e\big(\rho^\e, u^\e|\hr,\hu\big)(t) \leq C_0 \bigg\{\e + \frac{1}{|\OO_\e|} \E^\e\big(\rho^\e, u^\e|\hr,\hu\big)(0) \bigg\}.
\end{equation}
\end{theorem}

Before giving the proof, let us remark that the $C^{r,\alpha}$ regularity ($r\geq 2$, $\alpha \in ]0,1[$) and the injectivity radius bound ${\rm inj}\,(\OO) \geq \iota_0>0$ of the fluid domain $\OO$ are essential geometric assumptions  for the theorem. Moreover, Eq.\,\eqref{assumption on rho hat} is a natural condition on the {\em classical} solution $\hr$ --- it has neither concentration nor vacuum. Also, one needs $\eta >0$ to apply the uniform Korn's inequality. It should be emphasised that $C_0$ is independent of $\e$ and $(\rho^\e, u^\e)$.

\begin{proof}
For simplicity, throughout the proof let us denote by $(\rho,u)\equiv (\rho^\e, u^\e)$. Let us consider the relative entropy between $(\hr, U_\e)$ and $(\rho, u)$ in $\OO_\e$:
\begin{equation}
\E^\e\big(\rho,u|\hr,U_\e\big) := \int_{\OO_\e} \bigg\{\frac{1}{2}\rho|u-U_\e|^2 + \Big[H(\rho)-H'(\hr)(\rho-\hr) - H(\hr)\Big] \bigg\}\,\dd x,
\end{equation}
where $U_\e$ denotes the canonical lift of $\hu$; see Definition \ref{def: canonical lift in O}. That is, 
\begin{equation}\label{def of Ue}
U_\e := (\id \oplus V_\e) (\hu) = (\hu, V_{\e,\hu})^\top,
\end{equation}
where $V_{\e, \hu}$ denotes the vertical component of the canonical lift of $\hx$:
\begin{equation}\label{V eps u hat}
V_{\e, \hu} := \proj\, U_\e,
\end{equation}
and $\proj:\OO_\e \map \e\F$ is the  natural vertical projection onto the fibres. As $U_\e$ converges to $\hu$ strongly in the topology of $\hu$ ({\it e.g.}, in $C^0_tC^2_x$ as required by $\Lambda$, and such convergence depends only on the geometry of $\OO$), our strategy is to establish the inequality for $\E^\e\big(\rho,u|\hr,U_\e\big)$ and then send $\e\map 0^+$. As $(\rho, u)$ is a suitable weak solution, we have the relative entropy inequality for almost every $t\in[0,T]$:
\begin{align}\label{relative entropy ineq, epsilon}
&\E^\e(\rho,u|\hr, U_\e) (t) + \int_0^t\int_{\OO_\e} \big[\sph(\na u)-\sph(\na U_\e)\big]:(\na u - \na U_\e)\,\dd x\dd t'\nonumber\\
&\qquad\qquad\qquad \leq  \E^\e(\rho,u|\hr, U_\e) (0) +\int_0^t \RR^\e(\rho,u,\hr,U_\e)(t') \,\dd t'.
\end{align}
The remainder term is given by
\begin{align}\label{Re, before simplifying}
\RR^\e(\rho,u,\hr,U_\e) &:= \int_{\OO_\e} \rho\,\p_t U_\e \cdot (U_\e-u)\,\dd x + \int_{\OO_\e} \rho\,\na U_\e : \Big[u \otimes (U_\e-u)\Big]\,\dd x \nonumber\\
&\qquad - \int_{\OO_\e} \sph(\na U_\e):(\na u - \na U_\e)\,\dd x + \int_{\OO_\e} \dv\, U_\e \Big[p(\hr)-p(\rho)\Big]\,\dd x\nonumber\\
&\qquad +\int_{\OO_\e} \bigg\{(\hr-\rho) \p_t H'(\hr) + \big(\hr U_\e-\rho u\big)\cdot \na H'(\hr)\bigg\}\,\dd x \nonumber\\
&=: \one'^\e + \two'^\e + \three'^\e + \four'^\e + \five'^\e.
\end{align}

Therefore, the proof is reduced to  estimating the above expression \eqref{Re, before simplifying}. In the sequel, we divide our estimates into five steps.

\smallskip
{\bf Step 1.} Let us first simplify $(\one'^\e + \two'^\e+\five'^\e)$. By the identity \eqref{H''}, one has
\begin{align}
\one'^\e + \two'^\e+\five'^\e &= \int_{\OO_\e}\bigg\{\rho\bigg( \p_t U_\e\cdot(U_\e-u) + u\cdot \na U_\e\cdot(U_\e-u) \bigg)\nonumber\\
&\qquad\qquad + \frac{\hr-\rho}{\hr} (\p_t\hr)p'(\hr) + \frac{\hr(U_\e-u)+(\hr-\rho)u}{\hr}\cdot \na \hr \,p'(\hr)\bigg\}\,\dd x.
\end{align}
where $a\cdot M\cdot b\equiv M\cdot (a\otimes b)$ for vectors $a,b\in\R^N$ and $N\times N$ matrix $M$. Then, adding and subtracting the second and the third terms $I^2, I^3$ in below, we get
\begin{align}\label{I, II, V}
\one'^\e + \two'^\e+\five'^\e &= I^1 +I^2 + I^3+I^4+I^5\nonumber\\
&:= \int_{\OO_\e} \bigg\{\rho\bigg( \p_t \big(U_\e-\hu\big) + \big(U_\e \cdot \na U_\e - \hu \cdot \hna\hu \big) \bigg)\cdot (U_\e-u)\bigg\}\,\dd x\nonumber\\
&\quad + \int_{\OO_\e} \rho\big(\p_t\hu + \hu \cdot \hna\hu \big)\cdot(U_\e-u)\,\dd x -\int_{\OO_\e} \rho\na U_\e:(U_\e-u)\otimes (U_\e-u)\,\dd x\nonumber\\
&\quad+ \int_{\OO_\e} \frac{\hr-\rho}{\hr} \,p'(\hr)\big(\p_t\hr + u\cdot \na \hr \big)\,\dd x + \int_{\OO_\e} (U_\e-u)\cdot \na \hr\,p'(\hr)\,\dd x.
\end{align}
Applying Eq.\,\eqref{limiting momentum eq, NS}, the limiting monemtum equation on $\M$ to $I^2$, we get
\begin{equation}\label{I2}
I^2 =\int_{\OO_\e} \frac{\rho}{\hr}\bigg\{ \mu\hD\hu +\big(\eta + \frac{N-2}{N}\mu\big)\hna\hdv\hu + \TT(A,\hu) - \hna p(\hr) \bigg\} \cdot (U_\e-u)\,\dd x.
\end{equation} 
On the other hand, in view of Proposition \ref{propn: conservation law} we have
\begin{equation}
\p_t\hr + \dv\, (\hr\, U_\e) =  0 \qquad \text{ in } \OO_\e,
\end{equation}
which allows us to simplify $I^4$:
\begin{equation}\label{I4}
I^4 = -\int_{\OO_\e} \frac{\hr-\rho}{\hr}\,\na p(\hr)\cdot (U_\e -u)\,\dd x - \int_{\OO_\e} (\hr-\rho)\,p'(\hr)\,\dv\,U_\e\,\dd x.
\end{equation}
Thus, Eqs.\,\eqref{I, II, V}\eqref{I2} and \eqref{I4} together imply that
\begin{align}
\one'^\e + \two'^\e+\five'^\e &= \int_{\OO_\e} \bigg\{\rho \p_t \big(U_\e-\hu\big) -\int_{\OO_\e} \rho\na U_\e:(U_\e-u)\otimes (U_\e-u)\,\dd x\nonumber\\
&\quad +\int_{\OO_\e} \frac{\rho}{\hr}\bigg\{ \mu\hD\hu+\big(\eta + \frac{N-2}{N}\mu\big)\hna\hdv\hu + \TT(A,\hu) \bigg\} \cdot (U_\e-u)\,\dd x \nonumber\\
&\quad -  \int_{\OO_\e} (\hr-\rho)\,p'(\hr)\,\dv\,U_\e\,\dd x.
\end{align}
This further simplifies the $\RR^\e$ term in Eq.\,\eqref{Re, before simplifying}:
\begin{align}\label{Re, 1}
\RR^\e(\rho, u, \hr, U_\e) &= \int_{\OO_\e} \dv\, U_\e \bigg\{p(\hr)-p(\rho)-(\hr-\rho) p'(\hr) \bigg\}\,\dd x \nonumber\\
&\quad - \int_{\OO_\e} \rho \na U_\e : \big(U_\e-u\big) \otimes \big(U_\e-u\big) \,\dd x + \int_{\OO_\e} \sph(\na U_\e): \na(U_\e -u)\,\dd x\nonumber\\
&\quad +\int_{\OO_\e} \frac{\rho}{\hr}\bigg\{ \mu\hD\hu+\big(\eta + \frac{N-2}{N}\mu\big)\hna\hdv\hu + \TT(A,\hu) \bigg\} \cdot (U_\e-u)\,\dd x\nonumber\\
&\quad +  \int_{\OO_\e} \rho \bigg\{\,\p_t \big(U_\e-\hu\big) + U_\e \cdot \na U_\e - \hu \cdot \hna \hu\bigg\}\cdot(U_\e-u) \,\dd x.
\end{align}

\smallskip
{\bf Step 2.} Now we analyse the term involving the stress tensor $\sph(\na U_\e)$.  Integrating $\three^\e=\int_{\OO_\e}\sph(\na U_\e):\na(U_\e-u)\,\dd x$ by parts and applying the boundary condition \eqref{slip bc on Oe}, Lemma \ref{lemma: geometry} and the definition of $U_\e$ in \eqref{def of Ue}, we get
\begin{equation}
\int_{\OO_\e}\sph(\na U_\e):\na(U_\e-u)\,\dd x = -\int_{\OO_\e} [\dv\,\sph(\na U_\e)]\cdot(U_\e-u)\,\dd x.
\end{equation}
Using the definition of $\sph$, one easily deduces
\begin{equation}
\dv\,\sph(\na U_\e) = \mu \Delta U_\e + \vis \na\dv\,U_\e.
\end{equation}
Hence, by decomposing $\OO_\e$ along the horizontal ($T\M$) and vertical (fibre) directions, we have
\begin{align}\label{xxx}
\dv\,\sph(\na U_\e)\cdot (U_\e-u) &= \mu \bigg\{\Delta_\F V_{\e, \hu} + \hD V_{\e, \hu} \bigg\} \cdot \proj (U_\e-u)+\mu \hD\hu\cdot \huu \nonumber\\
&\quad + \vis \bigg\{\na_\F\dv_\F V_{\e,\hu}+\na_\F\hdv\hu\bigg\}\cdot \proj (U_\e-u)\nonumber\\
&\quad +\vis \bigg\{\hna\dv_\F V_{\e, \hu} + \hna\hdv\hu\bigg\}\cdot \huu.
\end{align}
Here we recall the notations: $\hna, \hdv, \hD$ are the covariant derivative, divergence and Laplace--Beltrami operators on $T\M$, and $\na_\F, \dv_\F, \Delta_\F$ are the corresponding differential operators on $\F$; $V_{\e, \hu}$ is defined in Eq.\,\eqref{V eps u hat}; in addition, $pr$ is the natural projection of $T\OO$ (or $T\OO_\e$) onto $T\F$ (or $T\F_\e$, resp.)

To proceed, by Lemma \ref{lemma: dlogA} and an obvious scaling, we find that $\Delta_\F V_{\e,\hu}=0$, $\na_\F\dv_\F V_{\e,\hu}=0$, $\na_\F\hdv\hu=0$ and $\dv_\F V_{\e, \hu}=\hna_{\hu}\log A$. So Eq.\,\eqref{xxx} becomes
\begin{align}
\dv\,\sph(\na U_\e)\cdot (U_\e-u) &= \mu \hD V_{\e,\hu}\cdot \proj (U_\e-u) + \mu\hD\hu\cdot\huu \nonumber\\
&+\vis \Big\{\hna_{\huu} \hna_{\hu}\log A + \hna_{\huu} \hdv \,\hu\Big\}.
\end{align}
Substituting back into Eq.\,\eqref{Re, 1}, one obtains:
\begin{align}\label{Re, 2}
\RR^\e(\rho, u, \hr, U_\e) &= \int_{\OO_\e} \dv\, U_\e \bigg\{p(\hr)-p(\rho)-(\hr-\rho) p'(\hr) \bigg\}\,\dd x \nonumber\\
&\quad - \int_{\OO_\e} \rho \na U_\e : \big(U_\e-u\big) \otimes \big(U_\e-u\big) \,\dd x \nonumber\\
&\quad +  \int_{\OO_\e} \rho \bigg\{\,\p_t \big(U_\e-\hu\big) + U_\e \cdot \na U_\e - \hu \cdot \hna \hu\bigg\}\cdot(U_\e-u) \,\dd x\nonumber\\
&\quad - \int_{\OO_\e} \bigg\{\mu \hD V_{\e,\hu}\cdot \proj (U_\e-u) +\mu\hD\hu\cdot\huu  \nonumber\\
&\qquad\qquad +\vis \Big(\hna_{\huu} \hna_{\hu}\log A + \hna_{\huu} \hdv \,\hu\Big)\bigg\}\,\dd x\nonumber\\
&\quad +\int_{\OO_\e} \frac{\rho}{\hr}\bigg\{ \mu\hD\hu+\big(\eta + \frac{N-2}{N}\mu\big)\hna\hdv\hu + \TT(A,\hu)\bigg\} \cdot \huu\,\dd x.
\end{align}
Therefore, noticing by Eq.\,\eqref{T term} that
\begin{equation*}
\TT(A,\hu)\cdot\huu = \vis \hna_{\huu}\hna_{\hu}\log A,
\end{equation*}
and that $V_{\e,\hu} = \hu$ for any $\e>0$ when restricted to $T\M$, we can rewrite Eq.\,\eqref{Re, 2} as follows:
\begin{align}\label{Re, 3}
\RR^\e(\rho, u, \hr, U_\e) &= \int_{\OO_\e} \dv\, U_\e \bigg\{p(\hr)-p(\rho)-(\hr-\rho) p'(\hr) \bigg\}\,\dd x \nonumber\\
&\quad - \int_{\OO_\e} \rho \na U_\e : \big(U_\e-u\big) \otimes \big(U_\e-u\big) \,\dd x \nonumber\\
&\quad +  \int_{\OO_\e} \rho \bigg\{\,\p_t \big(U_\e-\hu\big) + U_\e \cdot \na U_\e - \hu \cdot \hna \hu\bigg\}\cdot(U_\e-u) \,\dd x\nonumber\\
&\quad -\mu\int_{\OO_\e} \hD V_{\e,\hu}\cdot \proj (U_\e-u)\,\dd x\nonumber\\
&\quad + \int_{\OO_\e}\frac{\rho-\hr}{\hr}  \bigg\{\mu\hD\hu\cdot\huu \nonumber\\
&\quad\quad+ \vis \Big(\hna_{\huu}\hna_{\hu}\log A+\hna_{\huu}\hdv\hu\Big) \bigg\}\,\dd x\nonumber\\
&=: \one^\e + \two^\e + \three^\e + \four^\e + \five^\e. 
\end{align}

\smallskip
{\bf Step 3.} Now, in view of Eq.\,\eqref{Re, 3} and the relative entropy inequality \eqref{relative entropy ineq, epsilon}, the convergence can be established by taking $\e \map 0^+$ in each of the terms $\one^\e$ --- $\five^\e$.  

\noindent
{\bf Estimate for $\one^\e$.} Let us set $Q(\rho):=p(\hr)-p(\rho)-(\hr-\rho) p'(\hr)$. 

For $\rho \in [0,\urho]$, we clearly have $|Q(\rho)|\leq C_4=C(\lrho,\urho,\gamma)$. Also, for $\rho \in ]\urho,\infty[$, since $\lim_{\rho\map\infty}{\rho^{1-\gamma}p'(\rho) }=p_\infty$, we have $|Q(\rho)|\leq C_5(1+\rho^\gamma)$, with $C_5=C(p_\infty, \urho,\gamma)$. On the other hand, as $U_\e := (\id \oplus V_\e)(\hu)$ and $\hu$ is the classical solution on $\M$, $\|\dv\, U_\e(t,\cdot)\|_{C^0(\M)} \leq C_6=C(\|\hu\|_{C^0_tC^1_x})$ for each $t\in [0,T]$.  Hence, in view of the coercivity condition \eqref{p2}, we have
\begin{equation}\label{one estimate}
\big|\one^\e(t)\big| \leq C_7\, \E^\e\big(\rho, u|\hr,U_\e\big)(t),
\end{equation}
where $C_7$ depends possibly on $\|\hu\|_{C^0_tC^1_x}$, $\lrho,\urho$, $\gamma$ and $p_\infty$.

\noindent
{\bf Estimate for $\two^\e$.} Again  $\|\na U_\e(t,\cdot)\|_{C^0(\M)}\leq C_6$. Also, the Cauchy--Schwarz inequality for $N \times N$ matrices leads to $|(U_\e-u)\otimes(U_\e-u)| \leq N |U_\e-u|^2$, which implies that
\begin{equation}\label{two estimate}
\big| \two^\e(t) \big| \leq C_8 \E^\e\big(\rho, u|\hr,U_\e\big)(t),
\end{equation}
with $C_8=C(\|\hu\|_{C^0_tC^1_x}, N)$.

\noindent
{\bf Estimate for $\three^\e$.} Now we make crucial use of the definition $U_\e := (\id \oplus V_\e)(\hu) = (\hu, V_{\e,\hu})^\top$ by noting that
\begin{equation}
U_\e - \hu = V_{\e,\hu}
\end{equation}
is purely vertical, {\it i.e.}, lies in rescaled fibres $\e\F$. Recall from Eq.\,\eqref{def of V epsilon} that 
\begin{equation*}
V_{\e,\hu}(\bullet) = \e V_{\hu}\big(\frac{\bullet}{\e}\big) \qquad\text{ on } \e\F,
\end{equation*}
where $\G(T\M)\ni\hu \mapsto V_{\hu} \in \G(T\F)$ is the vertical component of the canonical lift of $\hu$ (see Definition \ref{def: canonical lift in O}). Hence
\begin{equation}\label{Ue-u}
\|U_\e(t,\cdot) - \hu(t,\cdot)\|_{C^0(\M)} \leq C_9\e \qquad \text{ for each } t\in[0,T],
\end{equation}
where $C_9=C(\OO)$ is determined by the geometry of the fluid domain $\OO$, independent of $\e$. Moreover, the classical solution $\hu$ satisfies $\|\p_t\hu(t,\cdot)\|_{C^0(\M)}\leq C_{10}=C(\|\hu\|_{C^0_tC^1_x})$. Thus
\begin{equation}
\big|\int_{\OO_\e} \rho\p_t(U_\e-\hu)\cdot (U_\e-u) \,\dd x\big| \leq C_{10}\e \int_{\OO_\e} \rho|U_\e-u|\,\dd x.
\end{equation}
Similarly, there exists $C_{11}=C(\|\hu\|_{C^0_tC^0_x},\OO)$ such that
\begin{equation}
\big|U_\e\cdot \na U_\e - \hu \cdot \hna \hu\big| \leq |U_\e||\na(U_\e-\hu)| + |U_\e-\hu||\hna\hu| \leq C_{11}\e\quad \text{ uniformly on } [0,T]\times\M.
\end{equation}
This gives us 
\begin{equation}
\big|\int_{\OO_\e} \rho (U_\e-u)\cdot \big(U_\e \cdot \na U_\e - \hu \cdot \hna \hu \big) \,\dd x\big| \leq C_{11}\e \int_{\OO_\e} \rho|U_\e-u|\,\dd x.
\end{equation}
In summary,
\begin{equation}\label{three estimate}
\big| \three^\e(t) \big| \leq (C_{10}+C_{11})\e \int_{\OO_\e} \rho|U_\e-u|\,\dd x.
\end{equation}

\noindent
{\bf Estimate for $\four^\e$.} The Laplacian of $V_{1,\hu}=V_{\hu}$ is uniformly bounded by a geometric constant. Hence, by scaling, there is some $C_{12}=C_{12}(\OO)$ independent of $\e$, such that
\begin{equation}\label{four estimate}
\big| \four^\e(t) \big| \leq C_{12} \mu \e \int_{\OO_\e}|U_\e-u|\,\dd x.
\end{equation}

\noindent
{\bf Estimate for $\five^\e$.} As in Bella--Feireisl--Lewicka--Novotn\'{y} \cite{bfln}, a cut-off function $\chi=\chi(\rho) \in C^\infty_c(]0,\infty[)$ is introduced:
\begin{equation}
0 \leq \chi \leq 1, \qquad \chi \equiv 1 \text{ on } [{\lrho}/{2}, 2\urho].
\end{equation}
Then we split $\five^\e$ into two terms:
\begin{align}\label{five split}
\big|\five^\e(t)\big| &\leq \big|\five_{\rm middle}^\e(t)\big| + \big|\five_{\rm ends}^\e(t)\big|\nonumber\\
&:= \Big| \int_{\OO_\e}\chi(\rho)\frac{\rho-\hr}{\hr}  \bigg\{\mu\hD\hu\cdot\huu \nonumber\\
&\quad\quad+ \vis \Big(\hna_{\huu}\hna_{\hu}\log A+\hna_{\huu}\hdv\hu\Big) \bigg\}\,\dd x\Big|\nonumber\\
&\quad + \Big| \int_{\OO_\e}\Big[1-\chi(\rho)\Big]\frac{\rho-\hr}{\hr}  \bigg\{\mu\hD\hu\cdot\huu \nonumber\\
&\quad\quad+ \vis \Big(\hna_{\huu}\hna_{\hu}\log A+\hna_{\huu}\hdv\hu\Big) \bigg\}\,\dd x\Big|.
\end{align}

For $\five_{\rm middle}^\e$, we directly estimate by Cauchy--Schwarz:
\begin{eqnarray}
&& \Big|\chi(\rho)\frac{\rho-\hr}{\hr} \mu\hD\hu\cdot\huu \Big| \leq C_{13} \mu\Big(|\rho-\hr|^2+|U_\e-u|^2\Big),\\&&\Big|\chi(\rho)\frac{\rho-\hr}{\hr}\vis\hna_\huu\hdv\hu\Big|\nonumber\\
&&\qquad\qquad\qquad\qquad\qquad \leq C_{13}\vis\Big(|\rho-\hr|^2+|U_\e-u|^2\Big),\\
&& \Big|\chi(\rho)\frac{\rho-\hr}{\hr} \vis \hna_{\huu}\hna_{\hu}\log A\Big| \nonumber\\
&&\qquad\qquad\qquad\qquad\qquad \leq C_{14}\vis\Big(|\rho-\hr|^2+|U_\e-u|^2\Big).
\end{eqnarray}
Here $C_{13}=C(\|\hu\|_{C^0_tC^2_x}, \lrho)$ and $C_{14}=C(\|\hu\|_{C^0_tC^1_x}, \|\hna\hna\log A\|_{C^0(\M)},\lrho)=C'(\|\hu\|_{C^0_tC^1_x}, \OO,\lrho)$. Thus, in view of Eq.\,\eqref{p1}, we have $C_{15}=C(\lrho,\urho,\|\hu\|_{C^0_tC^2_x},\OO)$ such that
\begin{equation}\label{five middle}
\big|\five_{\rm middle}^\e(t)\big| \leq C_{15} \vis \E^\e \big(\rho, u  |\hr, U_\e\big).
\end{equation}

Next, let us estimate $\five_{\rm middle}^\e$. Again by Cauchy--Schwarz there holds 
\begin{align}
&\bigg|\Big[1-\chi(\rho)\Big]\frac{\rho-\hr}{\hr}  \bigg\{ \mu\hD\hu\cdot\huu + \vis \hna_{\huu}\hdv \hu\bigg\}\bigg|\nonumber\\
&\qquad\qquad\qquad\leq C_{16}\vis\Big[1-\chi(\rho)\Big] \bigg\{\delta^{-1}  |\rho-\hr| + \delta(1+\rho)|U_\e-u|^2 \bigg\},
\end{align}
where $C_{16}=C(\|\hu\|_{C^0_tC^2_x}, \lrho)$ and $\delta>0$ is a small positive constant to be specified. Similarly, there exists $C_{17}=C(\|\hu\|_{C^0_tC^1_x}, \|\hna\hna\log A\|_{C^0(\M)},\lrho)=C'(\|\hu\|_{C^0_tC^1_x}, \OO,\lrho)$ such that 
\begin{eqnarray}
&&\Big|\Big[1-\chi(\rho)\Big]\frac{\rho-\hr}{\hr} \vis \hna_{\huu}\hna_{\hu}\log A \Big|\nonumber\\
&&\qquad\qquad\qquad \leq C_{17}\Big[1-\chi(\rho)\Big] \vis \bigg\{\delta^{-1}  |\rho-\hr| + \delta(1+\rho)|U_\e-u|^2 \bigg\}.
\end{eqnarray}
Here the assumption \eqref{injectivity rad} on the injectivity radius is crucial: together with the $C^{r,\alpha}$ regularity of $\OO$, it guarantees $\|\hna\hna\log A\|_{C^0(\M)} <\infty$ (see Eq.\,\eqref{poisson eq} in Proposition \ref{lemma: dlogA}).  Therefore, by Eq.\,\eqref{p2}, there is a constant $C_{18}=C(\lrho,\urho,\|\hu\|_{C^0_tC^2_x},\OO)$ such that
\begin{align}\label{five end}
\big|\five_{\rm end}^\e(t)\big| 
&\leq C_{18}\delta \vis \E^\e\big(\rho,u|\hr, U_\e\big) + C_{18}\vis  \times\nonumber\\
&\qquad\qquad\times\bigg\{\delta \int_{\OO_\e}\big|\big[1-\chi(\rho)\big](U_\e-u)\big|^2 \,\dd x + \frac{1}{\delta}\int_{\OO_\e}\big|\big[1-\chi(\rho)\big](\rho-\hr)\big|\,\dd x \bigg\}.
\end{align}
Furthermore, thanks to Young's inequality and Eq.\,\eqref{p2},
\begin{align}\label{five remainder}
\int_{\OO_\e}\big|\big[1-\chi(\rho)\big](\rho-\hr)\big|\,\dd x &\leq \urho |\OO_\e| + \frac{1}{\gamma} \int_{\OO_\e} \rho^\gamma\,\dd x + \frac{\gamma-1}{\gamma} |\OO_\e|\leq C_{19}\, \E^\e\big(\rho,u|\hr, U_\e\big),
\end{align}
where $C_{19}=C(\gamma,\lrho,\urho,\OO)$. Here we need the assumption that $\M$ and $\F$ are compact.  Finally, putting together Eqs.\,\eqref{five split}\eqref{five middle} \eqref{five end} and \eqref{five remainder}, one deduces
\begin{align}\label{five estimate}
\big|\five^\e(t)\big| 
&\leq C_{20} \vis\bigg\{ \E^\e\big(\rho,u|\hr, U_\e\big) +  \delta \int_{\OO_\e}|U_\e-u|^2 \,\dd x +\delta^{-1}|\OO_\e|\bigg\},
\end{align}
where $C_{20}=C(\gamma, \lrho,\urho,\|\hu\|_{C^0_tC^2_x},\OO)$. 

\smallskip
{\bf Step 4.}  In view of Eqs.\,\eqref{one estimate}\eqref{two estimate}\eqref{three estimate}\eqref{four estimate} and \eqref{five estimate} from Step 3 above, the relative entropy inequality \eqref{relative entropy ineq, epsilon} now reads
\begin{align}
&\E^\e\big(\rho, u|\hr, U_\e\big)(t)-\E^\e\big(\rho, u|\hr, U_\e\big)(0) + \int_0^t\int_{\OO_\e} \big[\sph(\na u)-\sph(\na U_\e)\big]:(\na u - \na U_\e)\,\dd x\dd t'\nonumber\\
\leq \,& C_{21}\bigg\{\int_{0}^t \E^\e\big(\rho, u|\hr, U_\e\big)(t') \,\dd t' + \e\int_{0}^t\int_{\OO_\e}\rho|U_\e-u|\,\dd x \dd t' + \mu\e\int_{0}^t \int_{\OO_\e}  |U_\e-u|\,\dd t' \dd x \nonumber\\
&\quad + \vis\bigg( \int_0^t \E^\e\big(\rho, u|\hr, U_\e\big)(t')\,\dd t' +\delta \int_0^t\int_{\OO_\e} |U_\e-u|^2\,\dd x \dd t' \bigg)\bigg\}
\end{align}
for almost every $t\in[0,T]$. The constant $C_{21}=C(N,\gamma,\lrho,\urho,p_\infty,\OO,\|\hu\|_{C^0_tC^2_x})$. 

Let us now establish a {\em claim}: For $\e_0,\delta>0$ sufficiently small, there holds
\begin{equation}\label{claim, delta term}
(\e_0+\delta) \int_0^t\int_{\OO_\e} |U_\e-u|^2\,\dd x \dd t' \leq \frac{1}{2\vis C_{21}} \int_0^t\int_{\OO_\e} \big[\sph(\na u)-\sph(\na U_\e)\big]:(\na u - \na U_\e)\,\dd x\dd t'.
\end{equation}
Indeed, for any $w \in W^{1,2}$ vector field on $\OO$ we have  
\begin{align}
\sph(\na w): \na w &= \mu\Big(2\mathbb{D}(\na w) - \frac{2}{N}\dv\, w\,\id\Big): \Big(\mathbb{D}(\na w)-\frac{1}{N}\dv\, w \,\id\Big)+\eta(\dv\, w)^2\nonumber\\
&= 2\mu\bigg|\mathbb{D}(\na w)-\frac{1}{N}\dv\, w \,\id\bigg|^2+\eta(\dv\, w)^2 \nonumber\\
&\geq C(\kappa) \mu \big|\mathbb{D}(\na w)\big|^2 + \kappa(\dv\, w)^2,
\end{align}
as $(2\mathbb{D}(\na w) - \frac{2}{N}\dv\, w\,\id)$ is a symmetric traceless $N\times N$ matrix, and $(\na w - [\mathbb{D}(\na w)-\frac{1}{N}\dv\, w \,\id])$ is the orthogonal projection of $\na w$ onto the annihilator of symmetric traceless $N\times N$ matrices. The last line holds by Cauchy--Schwarz applied to matrices, where $\kappa=\kappa(\mu,\eta, N)>0$; we note that $\eta>0$ is crucial here. Now, as it is assumed that $\OO_\e$ supports the uniform Korn's inequality, we can find a purely geometric constant $C_{22}=C(\OO)$, independent of $\e$, such that
\begin{equation}
\int_{\OO_\e} |w|^2\,\dd x\leq C_{22} \int_{\OO_\e}|\mathbb{D}(\na w)|^2\,\dd x. 
\end{equation}
Therefore, the desired constants $\delta$ and $\e_0$ in Eq.\,\eqref{claim, delta term} exist, which depend only on $C_{22}, C_{21}$, $\eta$,$\mu$ and $N$. Thus the {\em claim} follows.  

\smallskip 
{\bf Step 5.} Finally, our arguments in Step 4 shows that, 
for almost every $t\in[0,T]$, \begin{align}
&\E^\e\big(\rho, u|\hr, U_\e\big)(t)-\E^\e\big(\rho, u|\hr, U_\e\big)(0) + \frac{1}{2}\int_0^t\int_{\OO_\e} \big[\sph(\na u)-\sph(\na U_\e)\big]:(\na u - \na U_\e)\,\dd x\dd t'\nonumber\\
&\qquad\qquad\qquad\leq  C_{23}\bigg\{\int_{0}^t \E^\e\big(\rho, u|\hr, U_\e\big)(t') \,\dd t' + \e\int_{0}^t\int_{\OO_\e}\rho|U_\e-u|\,\dd x \dd t'\bigg\},
\end{align}
where $C_{23}=C(N,\mu,\eta,\gamma,\lrho,\urho,p_\infty,\OO,\|\hu\|_{C^0_tC^2_x})$ and $\e<\e_0$ as in the {\em claim} \eqref{claim, delta term}. In addition, notice that for some $C_{24}=C(\gamma,\lrho,\urho)$, we have 
\begin{align}\label{yyy}
&\int_{0}^t\int_{\OO_\e}\rho|U_\e-u|\,\dd x \dd t' \nonumber\\
\leq\,& 2\bigg\{\int_0^t \int_{\OO_\e}\rho\,\dd x\dd t' + \int_0^t\int_{\OO_\e} \rho|U_\e-u|^2\,\dd x\dd t'\bigg\}\nonumber\\
\leq \,& 2\int_0^t\E^\e\big(\rho,u|\hr,U_\e\big)(t')\,\dd t' + 2t|\OO_\e| + \frac{2}{\gamma}\int_0^t \int_{\OO_\e} \big[1-\chi(\rho)\big]\rho^\gamma\,\dd x\dd t' + \frac{2(\gamma-1)}{\gamma}\,t|\OO_\e|^{\frac{\gamma-1}{\gamma}}\nonumber\\
\leq\,& C_{24} \int_0^t\E^\e\big(\rho,u|\hr,U_\e\big)(t')\,\dd t' + 4t |\OO_\e| \qquad \text{ whenever } |\OO_\e|\leq 1, \, t \in [0,T].
\end{align}
Indeed, the second line follows from Cauchy--Schwarz, the third line from Young's inequality, H\"{o}lder's inequality and the cut-off function $\chi=\chi(\rho)$ as in Step 3 (similar to Eq.\,\eqref{five remainder}), and the last line from Eq.\,\eqref{p2}. Therefore,
\begin{align}\label{final estimate for NS convergence}
&\E^\e\big(\rho,u|\hr,U_\e\big)(t) +  \frac{1}{2}\int_0^t\int_{\OO_\e} \big[\sph(\na u)-\sph(\na U_\e)\big]:(\na u - \na U_\e)\,\dd x\dd t'\nonumber\\
&\qquad\qquad\leq \E^\e\big(\rho,u|\hr,U_\e\big)(0) + C_{25}\int_0^t \E^\e\big(\rho,u|\hr,U_\e\big)(t')\,\dd t' + C_{25}\e|\OO_\e|
\end{align}
for almost all $t\in[0,T]$, where $C_{25}$ depends  on $N$, $\mu$, $\eta$, $\gamma$, $\lrho$, $\urho$, $p_\infty$, $\Lambda$, the lifespan $T$ of the solutions to Eqs.\,\eqref{continuity eq on Oe}--\eqref{initial data for Oe} and Eqs.\,\eqref{limiting continuity eq, NS} -- \eqref{limiting initial data, NS}, as well as the geometry of $\OO$, but is independent of $\e$. The proof 
is now complete, in view of the Gr\"{o}nwall's inequality.     \end{proof}

An immediate corollary to Theorem \ref{thm: convergence of NS} is the following ``weak-strong stability'' theorem: the classical solution $(\hr,\hu)$, whenever it exists, is stable in the class of suitable weak solutions of compressible Navier--Stokes equations:
\begin{corollary}\label{Cor: weak strong unique, NS}
Let $\OO$, $(\rho^\e, u^\e)$ and $(\hr,\hu)$ be as in Theorem \ref{thm: convergence of NS}. Assume at $t=0$ that
\begin{equation}\label{convergence}
\dashint_{\OO_\e} \bigg\{\frac{1}{2} \rho^\e|u^\e-\hu|^2+\Big[H(\rho^\e)-H(\hr)-H'(\hr)(\rho^\e-\hr)\Big] \bigg\}\,\dd x  \longrightarrow 0 \qquad \text{ as } \e \map 0.
\end{equation}
Then the convergence in \eqref{convergence} holds for almost every $t\in[0,T]$.
\end{corollary}

\subsection{Dimension reduction limit and the vanishing viscosity limit of the Navier--Stokes system} 

In this subsection we establish the second main result of the paper, which answers the question {\bf (P--Euler)} in the affirmative. In this case we do not need the assumptions on the strict positivity of $\eta$ or the uniform Korn's inequality:

\begin{theorem}\label{thm: Euler convergence}
	Let $\OO \cong \M \times \F \subset \R^N$ be a fluid domain as in Sect.\,$1$. Let $(\rho^\e,u^\e):[0,T]\times \OO_\e \map \R_+ \times T\OO_\e$ be a family of suitable weak solutions to the Navier--Stokes equations \eqref{continuity eq on Oe}--\eqref{initial data for Oe}, indexed by $\e \map 0^+$, whose pressure term $p$ satisfies \eqref{assumption on p}. Suppose that Eqs.\,\eqref{limiting continuity eq, NS}--\eqref{limiting initial data, NS} have a classical solution $(\hr, \hu):[0,T]\times \M \map \R_+ \times T\M$, which satisfies
	\begin{equation}
	\Lambda:={\rm ess \, sup}_{t\in[0,T]}\, \big\|\hu(t,\cdot)\big\|_{C^2(\M)} < \infty,
	\end{equation}
as well as
\begin{equation}\label{assumption on rho hat'}
0<\lrho \leq \hr(t,x) \leq \urho <\infty\qquad \text{ on } [0,T]\times\M
\end{equation}
for some constants $\lrho,\urho$.  Then the following holds: There exists a constant $C_0$, depending only on $\Lambda$, $T$ and the geometry of $\OO$, such that for {\it a.e.} $t\in ]0,T]$, we have
\begin{equation}
\frac{1}{|\OO_\e|} \E^\e\big(\rho^\e, u^\e|\hr,\hu\big)(t) \leq C_0 \bigg\{\mu+\eta+\e + \frac{1}{|\OO_\e|} \E^\e\big(\rho^\e, u^\e|\hr,\hu\big)(0) \bigg\}.
\end{equation}

\end{theorem}

\begin{proof}
	The proof is mostly analogous to, and in many places simpler than, that of Theorem \ref{thm: convergence of NS}. Let us only emphasise  the differences. 
	
	First, by the same arguments as in Step 1 of the proof of Theorem \ref{thm: convergence of NS}, we can deduce
\begin{align}\label{Re 1 for Euler}
\RR^\e(\rho, u, \hr, U_\e) &= \int_{\OO_\e} \dv\, U_\e \bigg\{p(\hr)-p(\rho)-(\hr-\rho) p'(\hr) \bigg\}\,\dd x \nonumber\\
&\quad - \int_{\OO_\e} \rho \na U_\e : \big(U_\e-u\big) \otimes \big(U_\e-u\big) \,\dd x\nonumber\\
&\quad +  \int_{\OO_\e} \rho \bigg\{\,\p_t \big(U_\e-\hu\big) + U_\e \cdot \na U_\e - \hu \cdot \hna \hu\bigg\}\cdot(U_\e-u) \,\dd x\nonumber\\
&\quad + \int_{\OO_\e} \sph(\na U_\e): \na(U_\e -u)\,\dd x.
\end{align}
It is the same as Eq.\,\eqref{Re, 1}, except that the viscosity term is absent. $U_\e$ is again the canonical lift (see Definition \ref{def: canonical lift in O}) of $\hu$ to $T\OO_\e$, and Proposition \ref{propn: conservation law} is utilised in the derivation of Eq.\,\eqref{Re 1 for Euler}.

	Next, we notice that the first three terms are identical to $\one^\e$, $\two^\e$ and $\three^\e$ in Eq.\,\eqref{Re, 3}. Thus, by Eqs.\,\eqref{one estimate}\eqref{two estimate} and \eqref{three estimate}, we have
	\begin{align}\label{euler conv 1}
	&\E^\e\big(\rho, u|\hr, U_\e\big)(t)-\E^\e\big(\rho, u|\hr, U_\e\big)(0) + \int_0^t\int_{\OO_\e} \big[\sph(\na u)-\sph(\na U_\e)\big]:(\na u - \na U_\e)\,\dd x\dd t'\nonumber\\
	&\qquad\qquad\leq C_{26} \bigg\{\int_{0}^t \E^\e\big(\rho, u|\hr, U_\e\big)(t') \,\dd t'+\e\int_0^t\int_{\OO_\e}\rho|U_\e-u|\,\dd x\dd t'\bigg\}\nonumber\\
	&\qquad\qquad\qquad\quad+\int_0^t \int_{\OO_\e} \sph(\na U_\e): \na(U_\e -u)\,\dd x\dd t',
	\end{align}
from some $C_{26}=C(\Lambda,\urho,\lrho,\gamma,p_\infty,N,\OO)$.
	
	In addition, the second term on the right-hand side of Eq.\,\eqref{euler conv 1} can be estimated as Eq.\,\eqref{yyy}, reproduced below:
	\begin{align}\label{euler conv 2}
&\int_0^t\int_{\OO_\e}\rho|U_\e-u|\,\dd x\dd t' \nonumber\\
&\qquad\leq C_{27} \int_0^t\E^\e\big(\rho,u|\hr,U_\e\big)(t')\,\dd t' + 4t |\OO_\e| \quad\text{ whenever } |\OO_\e|\leq 1, \, t \in [0,T],
\end{align}	 
where $C_{27}=C(\gamma,\lrho,\urho)$. 
	
	Finally, we can dominate the last term in Eq.\,\eqref{euler conv 1} by the left-hand side. Invoking again the following identity in Step 4, proof of Theorem \ref{thm: convergence of NS}:
	\begin{equation}\label{S identity}
	\sph(\na w): \na w = 2\mu\Big|\mathbb{D}(\na w) - \frac{1}{N}\dv\, w\,\id\Big|^2+\eta(\dv\, w)^2,
	\end{equation}
where we take $w=U_\e-u$ here. Thus, we can bound
\begin{align}\label{euler conv 3}
&\Big|\int_0^t \int_{\OO_\e} \sph(\na U_\e): \na(U_\e -u)\,\dd x\dd t'\Big|\nonumber\\
=\,& \Big|\int_0^t \int_{\OO_\e}  \bigg\{\sqrt{2\mu}\Big(\mathbb{D}(\na U_\e)-\frac{1}{N}\dv\, U_\e\,\id\Big): \sqrt{\frac{\mu}{2}}\Big(\mathbb{D}(\na w)- \frac{1}{N}\dv\, w\,\id \Big) \nonumber\\
&\qquad + \sqrt{2\eta}\,\dv\,U_\e\sqrt{\frac{\eta}{2}}\,\dv \, w \bigg\}\,\dd x\dd t'\Big|\nonumber\\
\leq\,& \int_0^t \int_{\OO_\e} \bigg\{\mu\, \big|\mathbb{D}(\na w)- \frac{1}{N}\dv\, w\,\id\big|^2 + \frac{\eta}{2} (\dv\, w)^2 \bigg\}\,\dd x\dd t' \nonumber\\
&\qquad +  \int_0^t \int_{\OO_\e} \bigg\{2\mu \big|\mathbb{D}(\na U_\e) - \frac{1}{N}\dv\, U_\e\,\id\big|^2 + 2\eta (\dv\, U_\e)^2 \bigg\}\,\dd x \dd t'\nonumber\\
=\,& \int_0^t \int_{\OO_\e} \frac{\sph(\na w): \na w}{2}\,\dd x \dd t' + 2\int_0^t \int_{\OO_\e} \bigg\{\mu \big|\mathbb{D}(\na U_\e) - \frac{1}{N}\dv\, U_\e\,\id\big|^2 + \eta (\dv\, U_\e)^2 \bigg\}\,\dd x \dd t'.
\end{align}
Indeed, in the second line we use the definition of $\sph$, in the third line Cauchy--Schwarz, and the  final line Eq.\,\eqref{S identity}. However, $U_\e$ is the canonical lift of the strong solution $\hu$, which verifies $\|U_\e(t,\cdot)\|_{W^{1,2}(\OO_\e)}\leq C_{28}:=C(\Lambda, \OO)$. Hence
\begin{align}\label{euler conv 4}
&\int_0^t \int_{\OO_\e} \bigg\{\mu \big|\mathbb{D}(\na U_\e) - \frac{1}{N}\dv\, U_\e\,\id\big|^2 + \eta (\dv\, U_\e)^2 \bigg\}\,\dd x \dd t'\leq  C_{29}(\mu+\eta)|\OO_\e|
\end{align}
for some $C_{29}=C(T,\Lambda,\OO,N)$. Putting together Eqs.\,\eqref{euler conv 1}\eqref{euler conv 2}\eqref{euler conv 3} and \eqref{euler conv 4}, we obtain
\begin{align}
&\E^\e\big(\rho,u|\hr,U_\e\big)(t) + \frac{1}{2}\int_0^t\int_{\OO_\e} \big[\sph(\na u)-\sph(\na U_\e)\big]:(\na u - \na U_\e)\,\dd x\dd t'\nonumber\\
&\qquad\leq \E^\e\big(\rho,u|\hr,U_\e\big)(0) + C_{30}\bigg\{(\e+\mu+\eta)|\OO_\e| + \int_{0}^t \E^\e\big(\rho, u|\hr, U_\e\big)(t') \,\dd t' \bigg\},
\end{align}
where  $C_{30}=C(\Lambda,\urho,\lrho,\gamma,p_\infty,N,\OO,T)$, independent of the parameter $\e$ and suitable weak solutions $(\rho^\e, u^\e)$. Now, an application of the Gr\"{o}nwall's inequality completes the proof.   \end{proof}

We also have the following weak-strong stability result, whose proof is immediate:
\begin{corollary}\label{Cor: weak strong unique, Euler}
Let $\OO$, $(\rho^\e, u^\e)$ and $(\hr,\hu)$ be as in Theorem \ref{thm: convergence of NS}. Assume
\begin{equation}\label{convergence Euler}
\dashint_{\OO_\e} \bigg\{\frac{1}{2} \rho^\e|u^\e-\hu|^2+\Big[H(\rho^\e)-H(\hr)-H'(\hr)(\rho^\e-\hr)\Big] \bigg\}\,\dd x  \longrightarrow 0 \qquad \text{ as } \e+\lambda+\mu \map 0^+
\end{equation}
at $t=0$. Then the convergence in \eqref{convergence Euler} holds for almost every $t\in[0,T]$.
\end{corollary}

\section{Further Remarks}
We conclude the paper by several remarks:

1. We have assumed in Sect.\,1 that  $\F_x \perp T_x\M$ for each $x\in\M$, {\it i.e.}, the fibres of $\OO_\e$ (which shrink to zero as $\e\map 0^+$) are perpendicular to the base manifold. This can be relaxed  by only requiring {\em transversality}, namely Eq.\,\eqref{transversality}, as long as the union of the fibre boundaries $\bigcup_{x\in\M}\p\F_x$ still forms  a $C^{r,\alpha}$ hypersurface in $\R^N$, with $r \geq 2$ and $\alpha\in]0,1[$. Under such assumptions, Theorems \ref{thm: convergence of NS} and \ref{thm: Euler convergence} remain valid. This is because the fluid domain $\OO$ can be transformed to the one with perpendicular fibres via a $C^{r,\alpha}$ diffeomorphism, which leaves invariant the proof of convergence. A possible physical model is a thin, ``slanted'' nozzle.

2. Our results may apply to more general fluid domains  than product manifolds, {\it i.e.},  trivial fibre bundles. For example, consider a ``twisted circular nozzle'', namely $\OO = \sph^1 \times D^2/\sim$, where $\sim$ is the relation obtained by changing the direction of the diameter of $D^2$ by $180^\circ$ as it travels along $\sph^1$ once. Thus the fluid boundary $\fs = $ M\"{o}bius strip.  Our analysis still holds, provided that the divergence on $\OO$ is understood as being defined on each orientable chart, and that the uniform Korn's inequality remain valid.

3. In this paper we have discussed the convergence from suitable weak solutions to the classical limiting  solutions, under the hypothesis that the classical solutions exist.   In fact, the existence of solutions --- weak or strong --- to Eqs.\,\eqref{limiting continuity eq, NS}--\eqref{limiting initial data, NS}, Eqs.\,\eqref{limiting continuity eq, EULER}--\eqref{limiting initial data, EULER}, or more generally,   multi-dimensional systems of conservation/balance laws, remains open in the large; {\it cf.} Dafermos \cite{d}. 

4. One of the geometric assumptions of Theorem \ref{thm: convergence of NS} is the uniform Korn's inequality in the limit $\e\map 0^+$, namely Eq.\,\eqref{uniform korn ineq}. It is not void; for example, it is valid for the models of longitudinal nozzles with cross-sections and the circular nozzles (see Sect.\,5 in \cite{bfln}). More generally, it is valid for a family of thin shells around smooth hypersurfaces; {\it cf.} Lewicka--M\"{u}ller \cite{lm}; also see  Boulkhemair--Chakib \cite{bc} and Ruiz \cite{ruiz}, among many others. The existence or absence of uniform Korn's inequality in the dimension reduction limit, {\it i.e.}, along with the Gromov-Hausdorff convergence of Riemannian manifolds $\M_\e \map \M_0$ so that $\dim \M_\e > \dim \M_0$, remains an interesting problem for further investigation.

\bigskip
\noindent
{\bf Acknowledgement}.
Part of this work has been done during the author's stay as a CRM--ISM post-doctoral fellow at the Centre de Recherches Math\'{e}matiques, Universit\'{e} de Montr\'{e}al and the Institut des Sciences Math\'{e}matiques. Siran Li would like to thank these institutions for their hospitality.


\begin{thebibliography}{99}

\bibitem{arnold}
V. Arnol'd, {Sur la g\'{e}om\'{e}trie diff\'{e}rentielle des groupes de Lie de dimension infinie et ses applications \`{a} l'hydrodynamique des fluids parfaits}, \textit{Ann. Inst. Fourier $($Grenoble$)$}, \textbf{16} (1966), 319--361.

\bibitem{bc}
A. Boulkhemair, A. Chakib, {On the uniform Poincar\'{e} inequality}, \textit{Comm. PDE.}, \textbf{32} (2007), 1439--1447.


\bibitem{bfln}
P. Bella, E. Feireisl, M. Lewicka and A. Novotn\'{y}, {A rigorous justification of the Euler and Navier--Stokes equations with geometric effects}, \textit{SIAM J. Math. Anal.}, \textbf{48} (2016), 3907--3930.

\bibitem{bfn}
P. Bella, E. Feireisl, and A. Novotn\'{y}, {Dimension reduction for compressible viscous fluids}, \textit{Acta Appl. Math.}, \textbf{134} (2014), 111--121.


\bibitem{bds}
Y. Brenier, C. De Lellis and L. Sz\'{e}kelyhidi Jr., {Weak-strong uniqueness for measure-valued solutions}, \textit{Commun. Math. Phys.}, \textbf{305} (2001), 351--361.

\bibitem{cg}
G.-Q. Chen, J. Glimm, {Global solutions to the compressible Euler equations with geometrical structure}, \textit{Commum. Math. Phys.}, \textbf{180} (1996), 153--193.

\bibitem{d}
C.~M. Dafermos, \textit{Hyperbolic Conservation Laws in Continuum Physics}, Volume 325, Grundlehren der Mathematischen Wissenschaften
, 2010.

\bibitem{d2}
C.~M. Dafermos, {The second law of thermodynamics and stability}, \textit{Arch. Ration. Mech. Anal.}, \textbf{70} (1979), 167--179.


\bibitem{em}
D.~G. Ebin, J. Marsden, {Groups of diffeomorphisms and the motion of an incompressible fluid}, \textit{Ann. of Math.}, \textbf{92}  (1970), 102--163.

\bibitem{f}
E. Feireisl, \textit{Dynamics of viscous compressible fluids}, Vol. 26, Oxford University Press, 2004.


\bibitem{fjn}
E. Feireisl, B.~J. Jin and A. Novotn\'{y}, {Relative entropies, suitably weak solutions, and weak-strong uniqueness for the compressible Navier--Stokes system}, \textit{J. Math. Fluid Mech.}, \textbf{14} (2012), 717--730.

\bibitem{fnp}
E. Feireisl, A. Novotn\'{y} and H. Petzeltov\'{a}, {On the existence of globally defined weak solutions to the Navier--Stokes equations of compressible isentropic fluids}, \textit{J. Math. Fluid Mech.}, \textbf{3} (2001), 358--392.

\bibitem{fns}
E. Feireisl, A. Novotn\'{y} and Y. Sun, {Suitable weak solutions to the Navier--Stokes equations of compressible viscous fluids}, \textit{Indiana Univ. Math. J.}, \textbf{60} (2011), 611--631.

\bibitem{germain}
P. Germain, {Weak-strong uniqueness for the isentropic compressible Navier--Stokes system}, \textit{J. Math. Fluid Mech.}, \textbf{13} (2011), 137--146.

\bibitem{lm}
M. Lewicka, S. M\"{u}ller, {The uniform Korn-Poincar\'{e} inequality in thin domains}, \textit{Ann. Inst. H. Poincar\'{e}.  C, Anal. Non lineaire}, \textbf{28} (2011), 443--469.

\bibitem{lw}
P.~G. LeFloch, M. Westdickenberg, \textit{Finite energy solutions to the isentropic Euler equations with geometric effects}, \textit{J. Math. Pures Appl.}, \textbf{88} (2017), 389--429.

\bibitem{lions}
P.-L. Lions, \textit{Mathematical Topics in Fluid Dynamics}, Vol. 2, \textit{Compressible Models}, Oxford Science Publication, Oxford, UK, 1998.


\bibitem{ruiz}
D. Ruiz, {A note on the uniformity of the constant in the Poincar\'{e} inequality}, \textit{Adv. Nonlinear Stud.}, \textbf{12} (2012), 889--903.

\bibitem{s}
A.~I. Shnirelman, {The geometry of the group of diffeomorphisms
and the dynamics of an ideal incompressible fluid}, \textit{Mat. Sb. $($NS$)$}, \textbf{128} (1985),  82--109.

\bibitem{uhlenbeck}
K. Wehrheim, \textit{Uhlenbeck compactness}, European Mathematical Society series of lectures in mathematics, Euro. Math. Soc., 2003.


\bibitem{w}
E. Wiedemann, {Weak-strong uniqueness in fluid dynamics}, \textit{Arxiv Preprint} (2017), arXiv:1705.04220.

\end{thebibliography}
\end{document}